\documentclass[graybox]{svmult}
\usepackage{amsmath}
\usepackage{mathptmx}       
\usepackage{helvet}         
\usepackage{courier}      
\usepackage{type1cm}        
\usepackage{makeidx}        
\usepackage{graphicx}       
\usepackage{multicol}        
\usepackage[bottom]{footmisc}
\usepackage{amssymb}
\usepackage{amsfonts}
\usepackage{dsfont}
\makeindex       
\spdefaulttheorem{condition}{Condition}{\bfseries}{\itshape}
\begin{document}
\title*{Adaptive Density Estimation on the Circle by Nearly-Tight Frames}
\author{Claudio Durastanti\thanks{%
This research is supported by DFG Grant n.2131%
} }
\institute{Claudio Durastanti \at Ruhr Universit\"at, D44780, Bochum\\ \email{claudio.durastanti@gmail.com}}
\maketitle
\abstract{This work is concerned with the study of asymptotic properties of nonparametric density estimates in the framework of circular data. The estimation procedure here applied is based on wavelet
thresholding methods: the wavelets used are the so-called
Mexican needlets, which describe a nearly-tight frame on the circle.
We study the asymptotic behaviour of the $L^{2}$-risk function
for these estimates, in particular its adaptivity, proving that its rate of convergence is nearly optimal.
\\ \\ \textit{AMS 2010 Subject Classification - primary: 62G07; secondary: 62G20, 65T60, 62H11.}}
\section{Introduction}
In this work, we aim to study nonparametric estimation of a density
function $F$ based on directional data, sampled over the unit circle $%
\mathbb{S}^{1}$, by thresholding techniques, focussing in particular on the adaptivity for the associated $L^2$, the so-called mean integrated squared error. The estimators are built over a wavelet system, namely the Mexican needlets, which describes a nearly tight frame on $\mathbb{S}^1$ and characterized by strong localization in the spatial domain. 
Directional data over $\mathbb{S}^{1}$ can be viewed as angles measured with respect to a
fixed starting point, the origin, and a fixed positive direction. They can be described
as a set of points $\left\{ X_{i},i=1,...,n\right\} $, lying on the
circumference of $\mathbb{S}^{1}$: for this reason, they are also called
circular data. Circular data are characterized by the $2\pi $%
-periodicity, which has led to the development of a huge set of circular
statistical methods, independently from the standard real-line statistics.
These investigations can be motivated also in view of the large number of
applications in many different fields, as for instance geophysics,
oceanography and engineering. The textbooks \cite{rao, fisher} can provide a
complete overview on this topic and further technical details (see also \cite%
{bhatta, silverman}), while some applications of interest can be found in 
\cite{alial, dimarzio, kato, klemklem, wu}.

\subsection{Overview}
In the recent years, the literature concerning density estimation
problems is becoming more and more abundant: in particular, we are referring
to the study of the adaptivity results for $L^{2}$-risks in the nonparametric framework.
Consider a function $F$ belonging to some scale of classes $\mathcal{F}_\alpha$, called nonparametric regularity class of functions and depending on a set of parameters $\alpha \in A$, and its estimator $\widehat{F}$: this estimator is said to be adaptive for the $L^2$-risk and for $\mathcal{F}_\alpha$, if for any $\alpha \in A$ there exists a constant
$c_\alpha$ such that 
\begin{equation*}
\left\| \widehat{F}-F \right\|^2_{L^2\left(\mathbb{S}^1\right)} \leq c_\alpha R_n \left(\widehat{F},\mathcal{F}_\alpha \right),
\end{equation*}
where $n$ is the number of sampled data and $R_n \left(\widehat{F},\mathcal{F}_\alpha\right)$ is, loosely speaking, the worst possible performance over $\mathcal{F}_\alpha$. it is said to be minimax if $R_n \left(\widehat{F},\mathcal{F}_\alpha\right) = \inf_{F}\sup_{\widehat{F}\in \mathcal{F}_\alpha}\left\| \widehat{F}-F \right\|^2_{L^2\left(\mathbb{S}^1\right)}$, where $F$ ranges over all measurable functions of the observations $\left\{X_i,i=1,\ldots,n\right\}$. \\
Nonparametric minimax estimation of unknown densities or regression
functions was presented in the seminal paper \cite{donoho}, see also \cite{donoho1}: in this work,
optimal minimax rates of convergence of the $L^{2}$-risk were obtained by
nonlinear wavelet estimators based on thresholding techniques. Since then on,
many applications were developed not only in Euclidean spaces but also in
more general manifolds: we suggest as textbook reference \cite{WASA}. As far
as data on the unit $q$-dimensional sphere $\mathbb{S}^{q}$ are concerned,
many of those researches have been developed by using the constructions of
second-generation wavelets on $\mathbb{S}^{q}$ named spherical needlets. The
spherical needlets, introduced in the literature by \cite{npw1, npw2},
feature properties fundamental to attain the minimax optimal rates of
convergence of the estimates, such as their concentration in both Fourier
and space domains: density estimation of directional data on $\mathbb{S}^{q}$
was presented in \cite{bkmpAoSb}, the analysis of nonparametric regression
on sections of spin fiber bundles on $\mathbb{S}^{2}$ by the means of spin
needlets was proposed in \cite{dgm} and, finally, nonparametric regression
estimators on the sphere based respectively on needlet block and global thresholding were studied in 
\cite{durastanti2} and \cite{durastanti6}.

\subsection{Motivations and comparisons with standard needlets}
The main result here established concerns nearly-optimal rates of convergence for the $L^2$-risk of nonparametric density
estimation based on wavelet coefficients on $\mathbb{S}^{1}$. The wavelets considered are
the so-called Mexican needlets, introduced on general compact manifolds in \cite{gm0, gm1, gm2, gm3}, see also \cite{gelpes, pesenson}. These wavelets are
known to enjoy very good localization properties in the real domain, as
described in details below in Section \ref{secnearly} (see also \cite%
{durastanti1}), while their support is not bounded in the harmonic domain,
on the contrary of standard needlets. Furthermore, while standard needlets
are built by using a set of exact cubature points and weights (cfr. \cite%
{npw1}), Mexican needlets are built over a set of points satisfying weaker
restrictions (see \cite{gm2} and Theorem \ref{nearlytight} below). Indeed,
Mexican needlets can be built over any partition over their spatial support with area
monotonically decreasing with the resolution level. In this sense,
statistical techniques adopting Mexican needlets are more immediately
applicable for computational developing: some examples of their practical
applications in the field of statistics can be found, for instance, in \cite{dll, durastanti5, lanmar2, mayeli, scodeller}. 
On the other hand, Mexican needlets lack
an exact reconstruction formula, so that the corresponding density
estimators are biased. The main purpose of this work is to show that
thresholding procedures built on Mexican needlets behave asymptotically as
those constructed with standard needlets (cfr. \cite{bkmpAoSb}), on the
other hand offering advantages both from the practical and the theoretical
points of view, such as the easier construction of the wavelets over
partitions on $\mathbb{S}^{1}$ and the stronger localization properties;
their bias is proved to be asymptotically negliglible (see Theorem \ref%
{theorembias} below and numerical evidence in Section \ref{secsimulation}.

\subsection{Statement of the main result}
Given a set of i.i.d. circular data $\left\{ X_{i},i=1,...,n\right\} ,$
distributed over $\mathbb{S}^{1}$ with density $F$, and the set of circular
Mexican needlets, $\left\{ \psi _{jq;s}\left( \theta \right) ,\theta \in 
\mathbb{S}^{1}\right\} $, whose definition and main properties will be given
below in Subsection \ref{subharmonic}, a threshold wavelet estimator $%
\widehat{F}$ for the density function is given by 
\begin{equation*}
\widehat{F}\left( \theta \right)
=\sum_{j=J_{0}}^{J_{n}}\sum_{q=1}^{Q_{j}}\zeta _{jq}\left( \tau _{n}\right) 
\widehat{\beta }_{jq;sK}\psi _{jq;sK}\left( \theta \right) \text{ , }\theta
\in \mathbb{S}^{1},
\end{equation*}%
where $\zeta _{jq}\left( \tau _{n}\right) $ denotes the threshold, $\widehat{%
\beta }_{jq;sK}$ the unbiased estimator of the wavelet coefficient
corresponding to $\psi _{jq;sK}\left( \theta \right) $, $K$ is the cut-off
frequency; further details can be found in Section \ref{secprocedure}. We choose Besov spaces, labelled by $\mathcal{B}_{m,t}^{r}$, as nonparametric regression class of functions,
 (cfr. Subsection \ref%
{subbesov}), so that Theorem \ref{theoremmain} will prove that%
\begin{equation}
\sup_{F\in \mathcal{B}_{m,t}^{r}}\mathbb{E}\left[ \left\Vert \widehat{F}%
-F\right\Vert _{L^{2}\left( \mathbb{S}^{1}\right) }^{2}\right] =O_{n}\left(
\log n\left( \frac{n}{\log n}\right) ^{-\frac{2r}{2r+1}}\right),
\label{Ogrande}
\end{equation}%
where $r$ is one of the smoothness parameters characterizing the Besov space. 
Observe that the results here obtained are consistent with the ones already existing
literature, cfr. for instance \cite{bkmpAoSb, donoho, donoho1,WASA}. We stress again
that the estimator $\widehat{F}$ is characterized by a bias due to the lack
of an exact reconstruction formula: the nearly-tightness of $\left\{ \psi
_{jq,s}\left( \theta \right) ,\theta \in \mathbb{S}^{1}\right\} $ assures
the bias to be negligible with respect the rate of convergence on the left
hand of \eqref{Ogrande}, since it is controlled by some parameters depending
on the number of observations $n$. All the details can be found in Theorem %
\ref{theorembias}. We stress again that the study of the asymptotic behaviour of the bias is one
of the most relevant results attained in this paper, because it represents
the main difference between density estimates here defined and the ones
built on standard needlets (see again \cite{bkmpAoSb}).

\subsection{Plan of the paper}
Section \ref{secnearly} introduces the circular Mexican needlets, their
main properties and a quick overview on circular Besov spaces. In Section \ref%
{secprocedure} we describe the nonparametric density estimates built on circular Mexican needlets, while Section \ref{secdensity} describes our main results: Theorem \ref{theoremmain}, concerns adaptivity of
the threshold density estimator $\widehat{F}$ and Theorem \ref%
{theorembias} exploits the upper bound for the bias of $\widehat{F}$. 
Section \ref{secsimulation} provides some numerical evidence, while in Section \ref{auxiliary} we collect
all the auxiliary results related to the two main theorems and some
ancillary results on circular Mexican needlets.

\section{Nearly-tight frames on the circle\label{secnearly}}
This section will provide details concerning the construction and properties of Mexican needlet frames over $\mathbb{S}^1$ and the definition of circular Besov spaces in terms of their approximation properties. 

\subsection{Harmonic analysis and circular Mexican needlets\label{subharmonic}}
In this subsection we will describe some results, already well-known in the
literature, related to Fourier analysis and the construction of the Mexican
needlets over the unit circle $\mathbb{S}^{1}$. More details on Fourier
analysis can be found, for instance, on the textbook \cite{steinweiss},
while Mexican needlets and, more in general, nearly-tight frames over
compact manifolds were introduced in the literature in \cite{gm0, gm1, gm2, gm3}, see also \cite{gelpes,pesenson}. Furthermore, we present also a
simplified statement of the localization property in the spatial domain for
Mexican needlets over $\mathbb{S}^{1}$, described more extensively in Lemma %
\ref{localization} (see also \cite{durastanti1, gm2}).\\
Let us denote by $L^{2}\left( \mathbb{S}^{1}\right) \equiv L^{2}\left( 
\mathbb{S}^{1},d\rho \right) $ the space of square integrable functions over
the circle with respect to the Lebesgue measure $\rho \left( d\theta \right)
=\left( 2\pi \right) ^{-1}d\theta $, on which we define the inner product as
follows: for $f,g\in L^{2}\left( \mathbb{S}^{1},d\rho \right) $%
\begin{equation*}
\left\langle f,g\right\rangle \equiv \left\langle f,g\right\rangle
_{L^{2}\left( \mathbb{S}^{1}\right) }=\int_{\mathbb{S}^{1}}f\left( \theta
\right) \overline{g\left( \theta \right) }\rho \left( d\theta \right) ,
\end{equation*}%
As well known in the literature, the set $\left\{ u_{k}\left( \theta \right)
,\theta \in \mathbb{S}^{1},k\in \mathbb{Z}\right\} $, $u_{k}\left( x\right)
=\exp \left( ik\theta \right) $, describes an orthonormal basis over $%
\mathbb{S}^{1}$, whereas the Fourier transform is given by%
\begin{equation*}
a_{k}=\left\langle f,u_{k}\right\rangle _{L^{2}\left( \mathbb{S}^{1}\right)
}=\frac{1}{2\pi }\int_{0}^{2\pi }f\left( \theta \right) \overline{%
u_{k}\left( \theta \right) }d\theta ,
\end{equation*}%
and the corresponding Fourier inversion is given by%
\begin{equation}
f\left( \theta \right) =\sum_{k\in \mathbb{Z}}a_{k}u_{k}\left( \theta
\right) \text{ , }\theta \in \mathbb{S}^{1}.
\label{fourier expansion}
\end{equation}%
Furthermore, $\left\{ u_{k}\left( \theta \right) ,\theta \in \mathbb{S}%
^{1},k\in \mathbb{Z}\right\} $ can be viewed as the eigenfunctions of the
circular Laplacian $\Delta $ corresponding to eigenvalues $-k^{2}$ (for more
details, see for instance \cite{marpecbook}). For $F\in L^{2}\left( \mathbb{S%
}^{1}\right) $, the quantity $\gamma _{k}$ is given by%
\begin{equation}
\gamma _{k}:=\left\vert a_{k}\right\vert ^{2}, \label{powerspectrum}
\end{equation}%
so that 
\begin{equation*}
\sum_{k\in \mathbb{Z}}\gamma _{k}=\sum_{k\in \mathbb{Z}}\left\vert
a_{k}\right\vert ^{2}=\left\Vert F\right\Vert _{L^{2}\left( \mathbb{S}%
^{1}\right) }^{2}.
\end{equation*}
\begin{remark}
\label{convergence}Since $\left\Vert F\right\Vert _{L^{2}\left( \mathbb{S}%
^{1}\right) }^{2}<\infty $, the sum $\sum_{k\in \mathbb{Z}}\gamma _{k}$ has
to converge, therefore 
\begin{align*}
&\lim_{\left\vert k\right\vert \rightarrow \infty }\gamma _{k} =0,
\\
&\lim_{\left\vert k\right\vert \rightarrow \infty }\left\vert
a_{k}\right\vert =0.
\end{align*}
\end{remark}
Let us now introduce the Mexican needlet system. Let the weight function $%
w_{s}:\mathbb{R}\mapsto \mathbb{R}_{+}$ be given by%
\begin{equation}
w_{s}\left( x\right) :=x^{s}\exp \left( -x\right) \text{ , }x\in \mathbb{R}%
,  \label{weight}
\end{equation}%
so that, from the Calderon formula and for $t\in \mathbb{R}_{+}$, it holds
that%
\begin{equation*}
e_{s}:=\int_{0}^{\infty }\left\vert w_{s}\left( tx\right) \right\vert ^{2}%
\frac{dx}{x}=\frac{\Gamma \left( 2s\right) }{2^{2s}},
\end{equation*}%
while (see \cite{gm2}) using the Daubechies' criterion leads us to%
\begin{equation*}
\Lambda _{B,s}m_{B}\leq \sum_{j=-\infty }^{\infty }\left\vert w_{s}\left(
tB^{-2j}\right) \right\vert ^{2}\leq \Lambda _{B,s}M_{B},
\end{equation*}%
where, for the scale parameter $B>1$, 
\begin{align*}
&\Lambda _{B,s}=e_{s}\left( 2\log B\right) ^{-1},\\
&M_{B}=\left(1+O_{B}\left( \left\vert B-1\right\vert ^{2}\log \left\vert B-1\right\vert
\right) \right),\\
& m_{B}=\left( 1-O_{B}\left( \left\vert B-1\right\vert
^{2}\log \left\vert B-1\right\vert \right) \right). 
\end{align*}
\begin{figure}[tbp]
\centering
\includegraphics[width=\textwidth]{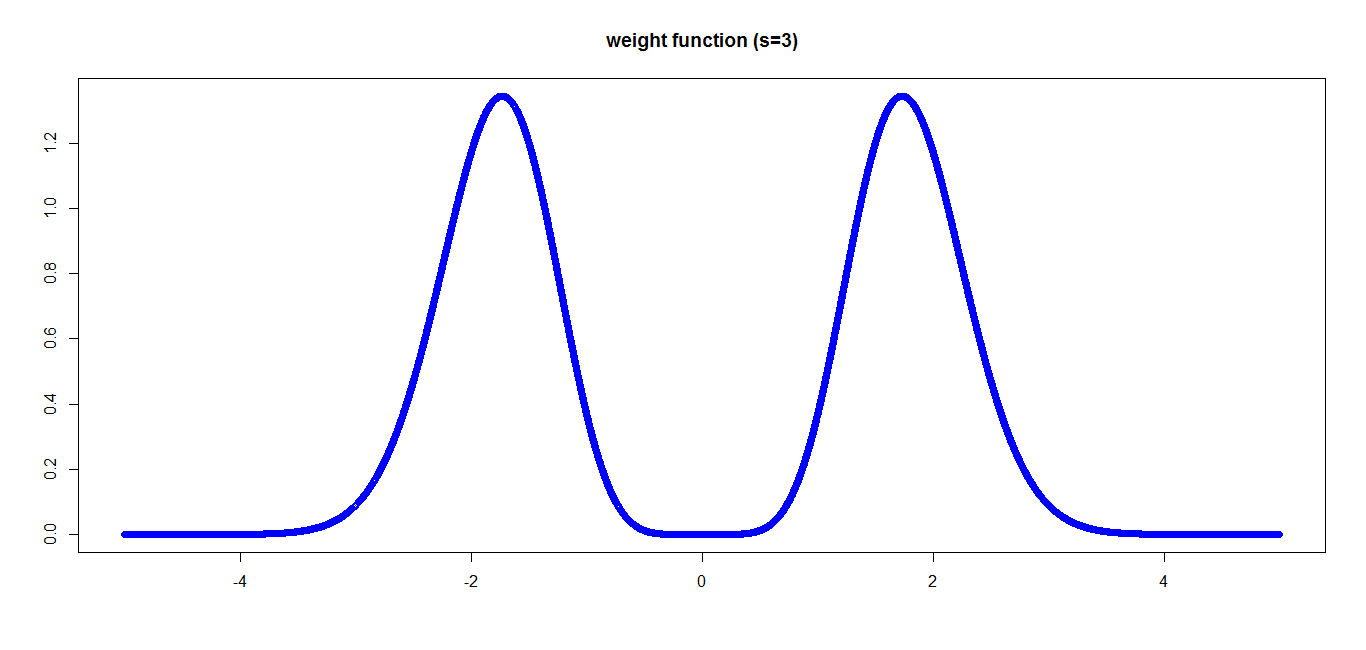}
\caption{the weight function $w_{s}\left(x^2\right)$ for $s=3$.}
\label{fig:1}
\end{figure}
\noindent For any resolution level $j\in \mathbb{Z}$, let $\left\{ E_{jq}\right\} ,$ $%
q=1,...,Q_{j}$ be a partition of $\mathbb{S}^{1}$, such that $E_{jq_{1}}\cap
E_{jq_{2}}=\varnothing $ for $q_{1}\neq q_{2}$. Any $E_{jq}$ is
characterized by the couple $\left( \lambda _{jq},x_{jq}\right) $: $\lambda
_{jq}=\rho \left( E_{jq}\right) $ describes the length of $E_{jq}$, while $%
x_{jq}\in E_{jq}$ is a point belonging to $E_{jq}$. For the sake of
simplicity, we can think to $x_{jq}$ as the midpoint of the segment of arc $%
E_{jq}$. Fixed now the shape parameter $s\in \mathbb{N}$ and the scale
parameter $B>1$, the circular Mexican needlet $\psi _{jq;s}:\mathbb{S}%
^{1}\mapsto \mathbb{C}$ is given by%
\begin{eqnarray}
\psi _{jq;s}\left( \theta \right)& := & \sqrt{\lambda _{jq}}\sum_{k=-\infty
}^{\infty }w_{s}\left( \left( B^{-j}k\right) ^{2}\right) \overline{%
u_{k}\left( x_{jq}\right) }u_{k}\left( \theta \right)  \notag \\
& = &\sqrt{\lambda _{jq}}\sum_{k=-\infty }^{\infty }w_{s}\left( \left(
B^{-j}k\right) ^{2}\right) \exp \left( ik\left( \theta -x_{jq}\right)
\right) \text{, }\theta \in \mathbb{S}^{1}.  \label{needletdef}
\end{eqnarray}
\begin{figure}[tbp]
\centering
\includegraphics[width=\textwidth]{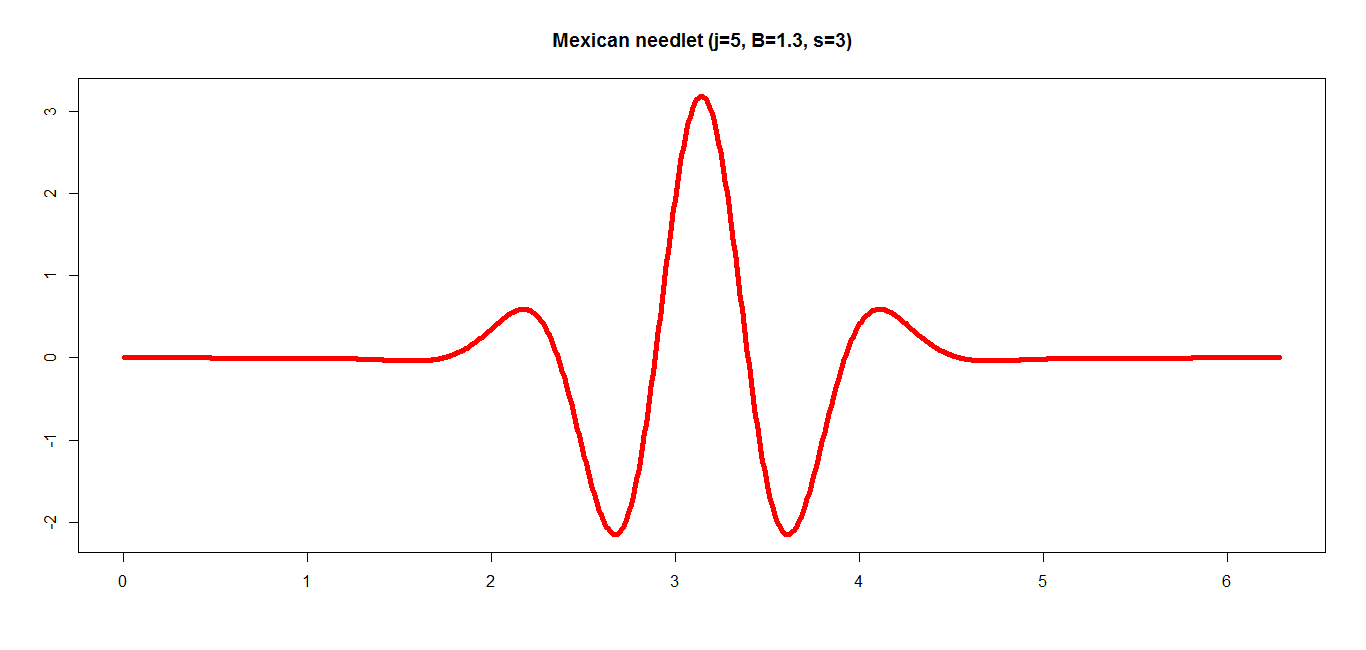}
\caption{The Mexican Needlet with $s=3, B=1.3, j=5$ centered on the point $%
x_{jq}=\protect\pi$.}
\label{fig:2}
\end{figure}
\noindent For any $F\in $ $L^{2}\left( \mathbb{S}^{1}\right) $, the needlet
coefficient $\beta _{jq;s}\in \mathbb{C}$ corresponding to $\psi _{jq;s}$ is
given by%
\begin{equation}
\beta _{jq;s}:=\left\langle F,\psi _{jq;s}\right\rangle _{L^{2}\left( 
\mathbb{S}^{1}\right) }.  \label{needletcoeff}
\end{equation}
The next result, here properly fitted for $\mathbb{S}^{1}$, was originally
proposed as Theorem 1.1 in \cite{gm2}: it proves that the Mexican needlet
framework describes a nearly-tight frame on the manifold. We recall that a
set of functions $\left\{ e_{i},i\geq 1\right\} $ defined over a manifold $M$
is a frame if there exist $c_{1},c_{2}>0$ so that, for any $F\in L^{2}\left(
M\right)$, 
\begin{equation*}
c_{1}\left\Vert F\right\Vert _{L^{2}\left( M\right) }^{2}\leq
\sum_{i}\left\vert \left\langle F,e_{i}\right\rangle _{L^{2}\left( M\right)
}\right\vert ^{2}\leq c_{2}\left\Vert F\right\Vert _{L^{2}\left( M\right)
}^{2}.
\end{equation*}%
A frame is said to be tight if $c_{1}=c_{2}$. An example of a tight frame
over the $d$-dimensional sphere $\mathbb{S}^{d}$ is given by the standard
needlets, introduced in the literature in \cite{npw1, npw2}. A frame is
nearly-tight if $c_{2}/c_{1}\simeq 1+\varepsilon $, where $\varepsilon $ is
close to $0$, cfr. \cite{gm1}.
\begin{theorem}
\label{nearlytight}(\textbf{Nearly-tightness of the Mexican needlets frame -
Th. 1.1 in \cite{gm2})} Fixing $B>1$ and $c_{0},\delta _{0}>0$ sufficiently
small, there exists a constant $C_{0}$ as follows:
\begin{itemize}
\item for $0<\eta <1$, suppose that for each $j\in \mathbb{Z}$, there exists
a set of measurable sets $\left\{ E_{jq},q=1,...,Q_{j}\right\} $, with $%
\lambda _{jq}=\mu \left( E_{jq}\right) $, where:
\begin{itemize}
\item $\lambda _{jq}\leq \eta B^{-j}$;
\item for each $j$ with $\eta B^{-j}<\delta _{0}$, $\lambda _{jq}\geq
c_{0}\left( \eta B^{-j}\right) $ for $q=1,...,Q_{j}$;
\end{itemize}
\item it holds that 
\begin{equation*}
\left( \Lambda _{B,s}m_{B}\!-\! C_{0}\eta \right) \left\Vert F\right\Vert
_{L^{2}\left( \mathbb{S}^{1}\right) }^{2}\!\leq \! \sum_{j=-\infty }^{\infty
}\sum_{q=1}^{Q_{j}}\left\vert \beta _{jq;s}\right\vert ^{2} \!\leq\! \left(
\Lambda _{B,s}M_{B}\!+\!C_{0}\eta \right) \left\Vert F\right\Vert _{L^{2}\left( 
\mathbb{S}^{1}\right) }^{2}.
\end{equation*}%
If $\left( \Lambda _{B,s}m_{B}-C_{0}\eta \right) >0$, then $\left\{ \psi
_{jq;s}\right\} $ is a nearly tight frame, since 
\begin{equation*}
\frac{\left( \Lambda _{B,s}M_{B}+C_{0}\eta \right) }{\left( \Lambda
_{B,s}m_{B}-C_{0}\eta \right) }\sim \frac{M_{B}}{m_{B}}=1+O_{B}\left(
\left\vert B-1\right\vert ^{2}\log \left\vert B-1\right\vert \right) .
\end{equation*}
\end{itemize}
\end{theorem}
Mexican needlets can be thought as an alternative approach to the
standard needlets, proposed in \cite{npw1,npw2}, see also \cite{bkmpAoSb,marpecbook}, in views of their stronger localization
property in the real domain. Standard needlets feature a
quasi-exponential localization property in the spatial range, while the
weight function $w_{s}$ leads to a full-exponential localization in the real
space as proved below in Lemma \ref{localization} (cfr. \cite{durastanti1,gm2}). As far as the frequency domain is concerned, while spherical
needlets lie on compact support (see again \cite{npw1,npw2}), each
Mexican needlet has to take in account the whole frequency range. This issue
is partially compensated by the structure itself of the function $w_{s}$,
exponentially localized around a dominant term in the frequency domain and,
therefore, consistently different from zero only on limited set of
frequencies. For our purposes and in order to respect the conditions in
Theorem \ref{nearlytight}, we impose the following
\begin{condition}
\label{areaandcardinality}. Let $\psi _{jk;s}\left( \theta \right) $ and $%
\beta _{jk;s}$ be given respectively by \eqref{needletdef} and \eqref{needletcoeff}. We have that, for $j>0$%
\begin{equation*}
Q_{j}\approx \eta ^{-1}B^{j}\text{ , }\lambda _{jq}\approx \eta B^{-j},
\end{equation*}%
so that 
\begin{equation}
\psi _{jq;s}\left( \theta \right) \approx \eta ^{\frac{1}{2}}B^{-\frac{j}{2}%
}\sum_{k=-\infty }^{\infty }w_{s}\left( \left( B^{-j}k\right) ^{2}\right)
\exp \left( ik\left( \theta -x_{jq}\right) \right) \text{ , }\theta \in 
\mathbb{S}^{1}.  \label{notation}
\end{equation}%
Furthermore, we choose $J_{0}<-\log _{B}\sqrt{s}$ and fix $\delta _{0}$ such
that $\delta _{0}\leq \eta B^{-J_{0}}$. Hence, we have, for $j<J_{0}$, 
\begin{equation}
Q_{j}=1\text{, }\lambda _{jq}=2\pi .  \label{jnegativo}
\end{equation}
\end{condition}
Mexican needlets are characterized by the following localization property,
proven in Lemma \ref{localization}:%
\begin{equation*}
\left\vert \psi _{jk;s}\left( \theta \right) \right\vert \leq \! \sqrt{\lambda
_{jk}}c_{s}B^{j}\exp \!\left( \!-\left( \!\frac{B^{j}\left( \theta -x_{jk}\right) 
}{2}\right) ^{2}\right) \!\left( 1+\!\left( \frac{B^{j}\left( \theta
-x_{jk}\right) }{2}\right) ^{2s}\right).
\end{equation*}
From the localization property, it follows a bound rule on the norms: there
exist $\widetilde{c_{p}},\widetilde{C_{p}}>0$ such that 
\begin{equation}
\widetilde{c_{p}}B^{j\left( \frac{p}{2}-1\right) }\eta ^{\frac{p}{2}}\leq
\left\Vert \psi _{jq;s}\right\Vert _{L^{p}\left( \mathbb{S}^{1}\right)
}^{p}\leq \widetilde{C_{p}}\eta ^{\frac{p}{2}}B^{j\left( \frac{p}{2}%
-1\right) }\text{ .}  \label{normbound}
\end{equation}%
The proof, totally analogous to the case of standard needlets (see \cite%
{npw2}), is here omitted.
\begin{remark}
\label{reasons}The choice of \eqref{jnegativo} is justified as follows.
First of all, observe that, for any $j<J_{0}$, $\lambda _{jq}\ $still
satisfies Theorem \ref{nearlytight}. Furthermore, when $j$ is negative, the $%
B^{-j}$ grows to infinity, hence there exists some $J^{\prime }<0$ such that 
$\delta _{0}\leq \eta B^{-J^{\prime }}$. It implies that the $\lambda _{jq}$
has to be smaller than a quantity bigger$\ $than $4\pi =\rho \left( \mathbb{S%
}^{1}\right) $, corresponding to the case $E_{jq}\equiv \mathbb{S}^{1}$,
which leads to $Q_{j}=1$, so that we have that $Q_{j}\lambda _{jk}\approx 1$%
. As far as the choice of $J_{0}$ is concerned, if $J_{0}<-\log _{B}\sqrt{s}$%
, it means that, for any $k\,$, $\left\vert kB^{-J_{0}}\right\vert >s$, and
therefore $w_{s}\left( \left( kB^{-J_{0}}\right) ^{2}\right) <$ $w_{s}\left(
s\right) =\max_{r\in \mathbb{R}}w_{s}\left( r\right) $. As consequence,
taking into account Lemma \ref{ecchecaspita}, we have that for any $k$, $\chi _{s,B,J_{0}}\left( k^{2}\right) <<2^{-2s}\Gamma \left( 2s\right) =e_{s}$.
\end{remark}
\begin{remark}
While in \cite{gm0,gm1,gm2,gm3} the Mexican needlets are defined as $\psi
_{j^{\prime }q}\left( \theta \right) \equiv \psi _{-jq}\left( \theta \right) 
$, $\theta \in \mathbb{S}^{1}$. We use this notation to uniform this work to
the already existing literature on the field of statistics based on
needlet-like framework.
\end{remark}

\subsection{Besov spaces on the circle\label{subbesov}}
In this subsection, we will recall some of the results proposed in \cite{gm3}
(see also \cite{WASA, npw2}) on Besov spaces, in terms of their approximation properties. 
More in details, let $\Pi _{r}$
be the space of polynomials of degree $r$: we start by looking for the
infimum of the $L^{p}\left( \mathbb{S}^{1}\right) $-distance between a
function $f:\mathbb{S}^{1}\mapsto \mathbb{R}$ and the space $\Pi _{r}$:%
\begin{equation*}
G_{r}\left( f,P\right) =\inf_{P\in \Pi _{r}}\left\Vert f-P\right\Vert
_{L^{p}\left( \mathbb{S}^{1}\right) }.
\end{equation*}%
Following, for instance, \cite{bkmpAoSb, dgm, gm3, npw2}, let $F\in \mathcal{%
B}_{m,t}^{r}$, if and only if both the following conditions hold:%
\begin{equation*}
\left( i\right) \text{ }F\in L^{m}\left( \mathbb{S}^{1}\right) \text{ , }%
\left( ii\right) \left( \sum_{u}\left( u^{s}G_{u}\left( f,P\right) \right) ^{%
\frac{r}{u}}\right) ^{\frac{1}{r}},
\end{equation*}
or, equivalently, 
\begin{equation*}
\left( i\right) \text{ }F\in L^{m}\left( \mathbb{S}^{1}\right) \text{ , }%
\left( ii\right) \left( \sum_{j}\left( B^{-jr}G_{B^{j}}\left( f,P\right)
\right) ^{q}\right) ^{\frac{1}{q}}.
\end{equation*}%
As shown in \cite{gm3}, see also \cite{bkmpAoSb}, it holds that, for $1\leq
m\leq \infty $, $r>0$, $0\leq t\leq \infty $, $f\in $ $\mathcal{B}_{m,t}^{r}$
if and only if%
\begin{equation*}
\left( \sum_{q=1}^{Q_{j}}\left\vert \beta _{jq;s}\right\vert ^{m}\left\Vert
\psi _{jk;s}\right\Vert _{L^{m}\left( \mathbb{S}^{1}\right) }^{m}\right) ^{%
\frac{1}{m}}<B^{-jr}\delta _{j},\text{ }\delta _{j}\in \ell _{r}.
\end{equation*}%
In what follows, we will make extensive use of this inequality with $m=2$:%
\begin{equation}
\left( \sum_{q=1}^{Q_{j}}\left\vert \beta _{jq;s}\right\vert ^{2}\left\Vert
\psi _{jk;s}\right\Vert _{L^{2}\left( \mathbb{S}^{1}\right) }^{2}\right) ^{%
\frac{1}{2}}\leq \left( \widetilde{C_{2}}\eta \sum_{q=1}^{Q_{j}}\left\vert
\beta _{jq;s}\right\vert ^{2}\right) ^{\frac{1}{2}}<B^{-jr}\delta _{j},\text{
}\delta _{j}\in \ell _{r}. \label{besovineq}
\end{equation}
Further details on Besov spaces can be found in \cite{npw2} and in the
textbook \cite{WASA}.

\section{The density estimation procedure\label{secprocedure}}
In this section, we will introduce a thresholding density function estimator
on the circle based on the Mexican needlet coefficients. As already mentioned, thresholding
techniques were introduced in the literature by D. Donoho and I. Johnstone
in \cite{donoho}, to be later successfully applied in several research topics: for
an exhaustive overview and details we suggest the textbooks \cite{WASA} and \cite{tsyb}.
Consider a set of random directional observations $\left\{ X_{i}\in \mathbb{S%
}^{1}:i=1,...,n\right\} $ with common distribution $v\left( \theta \right)
=F\left( \theta \right) d\theta $ and let us introduce the threshold function $\zeta
_{jq}\left( \tau _{n}\right) :=\mathds{1}_{\left\{ \left\vert \beta
_{jq;s}\right\vert \geq \kappa \tau _{n}\right\} }$, where $\kappa $ is a
real-valued positive constant to be chosen to set the size of the threshold
(cfr. \cite{bkmpAoSb}). The coefficient estimator is given by 
\begin{equation*}
\widehat{\beta }_{jq;s}:=\frac{1}{n}\sum_{i=1}^{n}\overline{\psi }%
_{jq;sK}\left( X_{i}\right),
\end{equation*}%
which is unbiased, i. e. 
\begin{equation*}
\mathbb{E}\left[ \widehat{\beta }_{jqK;s}\right] =\int_{\mathbb{S}^{1}}%
\overline{\psi }_{jq;sK}F\left( \theta \right) d\theta =\beta _{jqK;s}.
\end{equation*}%
Consequently, the thresholding density estimator is given by 
\begin{equation}
\widehat{F}\left( \theta \right)
=\sum_{j=J_{0}}^{J_{n}}\sum_{q=1}^{Q_{j}}\zeta _{jq}\left( \tau _{n}\right) 
\widehat{\beta }_{jqK;s}\psi _{jq;sK}\left( \theta \right) \text{ , }\theta
\in \mathbb{S}^{1}, \label{estimatordef}
\end{equation}%
where $J_{n}$ and $K_{n}$ represent respectively the truncation resolution level and
the cut-off frequency. The truncation level is chosen so that $B^{J_{n}}=\sqrt{\frac{n}{\log n}}$, as usual in the literature (see for instance \cite%
{bkmpAoSb, dgm}), while the cut-off frequency is fixed so that $K_{n}=%
\sqrt{\frac{n}{\log n}}$, . The other tuning parameters of the Mexican needlet
estimator to be considered are:
\begin{itemize}
\item the threshold constant $\kappa $, whose evaluation is given in the
Section 6 of \cite{bkmpAoSb};
\item the scaling factor $\tau _{n}$, depending on the sample size,
chosen, as usual in the literature, as $\tau _{n}=\left( \log n/n\right)
^{1/2} $;
\item the pixel-parameter $\eta _{n}=\eta $, chosen so that $\eta
_{n}=O_n\left( n^{-\frac{2}{3}}\right)$.
\end{itemize}
We will present our main result concerning Mexican thresholding density
estimation in the next Theorem. For the embeddings featured by the Besov
spaces, as in \cite{bkmpAoSb}, the condition $r>\frac{1}{m}$ implies that $%
F\in \mathcal{B}_{m,t}^{r}\subset \mathcal{B}_{\infty ,t}^{r-\frac{1}{m}}$,
so that $F$ is continuous.
\begin{theorem}
\label{theoremmain}For $1\leq m=t<2$, $r>\frac{1}{m}$, there exists some
constant $C_{0}=C_{0}\left( m,r\right) $ such that%
\begin{equation}
\sup_{F\in \mathcal{B}_{m,t}^{r}}\mathbb{E}\left[ \left\Vert \widehat{F}%
-F\right\Vert _{L^{2}\left( \mathbb{S}^{1}\right) }^{2}\right] \leq
C_{0}\log n\left( \frac{n}{\log n}\right) ^{-\frac{2r}{2r+1}}\text{ .}
\label{mainstatement}
\end{equation}
\end{theorem}
\begin{remark}
To attain optimality, it should be necessary to show also that 
\begin{eqnarray*}
\sup_{F\in \mathcal{B}_{m,t}^{r}}\mathbb{E}\left[ \left\Vert \widehat{F}-F\right\Vert _{L^{2}\left( \mathbb{S}^{1}\right) }^{2}\right] \geq C_{*} \left( \frac{n}{\log n}\right) ^{-\frac{2r}{2r+1}}.
\end{eqnarray*}
This lower bound is entirely analogous to the standard needlet case in \cite{bkmpAoSb}, Theorem 11, and therefore its proof is here omitted.
\end{remark}

\section{Proof of Theorem \protect\ref{theoremmain}\label{secdensity}}
In this section we will provide a proof for Theorem \ref{theoremmain} based
on the main guidelines described by D. Donoho and I. Johnstone in \cite%
{donoho}, cfr. also \cite{bkmpAoSb} and the textbooks \cite{WASA,tsyb}. The procedure illustrated by \cite{bkmpAoSb, donoho, donoho1} fits perfectly for tight
wavelet systems, which feature an exact reconstruction formula. As already
discussed in Subsection \ref{subharmonic}, Mexican needlet are not
characterized by tightness, hence the bias term appearing in the study of (%
\ref{mainstatement}) will also take into account addends due the (deterministic) error raising when we approximate a function with its wavelet expansion. The decay of these terms will depend on the choice
of the pixel-parameter $\eta _{n},$ on one hand, and of $J_{n}$ and $K_{n}$ on the other hand.
We will start by developing an upper bound for the bias term, which represents the
main difference between the estimation procedure here discussed and the one based on
standard needlet frames.

\subsection{The bias: the construction and the upper bound}
We recall from \cite{gm2} the so-called summation operator $S$, leading to
the \textit{summation formula}. The summation formula can be viewed as the
equivalent in the Mexican needlet framework of the reconstruction formula in
the standard needlet case (see for instance \cite{npw1, marpecbook}): for
any $F\in L^{2}\left( \mathbb{S}^{1}\right) $, let the \textit{summation
operator} $S\left[ F\right] _{s}$ be given by%
\begin{equation}
S\left[ F\right] _{s}\left( \theta \right) :=\sum_{j=J_{0}}^{\infty
}\sum_{q=1}^{Q_{j}}\beta _{jq;s}\psi _{jq;s}\left( \theta \right) \text{ , }%
\theta \in \mathbb{S}^{1}\text{.}  \label{summationoperator}
\end{equation}%
The goal of this subsection is also to estimate which terms in the sum above
are so small that they can be neglected. We will fix a cut-off frequency $K$%
, to compensate the lack of a compact support in the harmonic domain typical
of standard needlets (see \cite{npw1}), to define the truncated Mexican
needlet,\ and a truncation resolution level $J$. Theorem \ref{theorembias} will
exploit an upper bound, depending on $s$, $J$, $K$ and $\eta $, between \eqref{summationoperator} and the truncated summation operator defined below. Observe that these results are general and not related to the specificity of the estimation problem: in this sense, when the label $n$ in $K$ and $J$ is omitted, we intend that the claimed result holds in general.\\
First of all, given $K\in \mathbb{N}$, the \textit{truncated Mexican needlet} $\psi _{jq;sK}$ is given by
\begin{equation*}
\psi _{jq;sK}\left( \theta \right) :=\sqrt{\lambda _{jq}}\sum_{\left\vert
k\right\vert \leq K}w_{s}\left( \left( kB^{-j}\right) ^{2}\right) \overline{%
u_{k}\left( \xi _{jq}\right) }u_{k}\left( \theta \right) \text{ , }\theta
\in \mathbb{S}^{1}\text{, }\xi _{jq}\in E_{jq}\text{ ,}
\end{equation*}%
and the corresponding \textit{truncated needlet coefficient} $\beta _{jq;sK}$
is defined%
\begin{equation*}
\beta _{jq;sK}:=\left\langle F,\psi _{jq;sK}\right\rangle _{L^{2}\left( 
\mathbb{S}^{1}\right) }\text{ .}
\end{equation*}%
Loosely speaking, fixed $K$, $\psi _{jk;sK}\left( \cdot \right) $ is the
Mexican needlet where all the elements out of the support $\left[ -K,K\right]
$ are not taken into account. The \textit{truncated summation operator} $S%
\left[ F\right] _{s,K,J}$ is therefore given by%
\begin{equation}
S\left[ F\right] _{s,K,J}\left( \theta \right)
:=\sum_{j=J_{0}}^{J}\sum_{q=1}^{Q_{j}}\beta _{jq;sK}\psi _{jq;sK}\left(
\theta \right) \text{ , }\theta \in \mathbb{S}^{1}\text{.}  \label{cutted}
\end{equation}
\begin{remark}
Following Remark \ref{reasons}, we will truncate in \eqref{cutted} all the
negative resolution levels $j<J_{0}$.
\end{remark}
Let the bias $R_{s,K,J,\eta }$ be given by 
\begin{equation}
R_{s,K,J,\eta }:=\left\Vert S\left[ F\right] _{s}-S\left[ F\right]
_{s,K,J}\right\Vert _{L^{2}\left( \mathbb{S}^{1}\right) }\text{ ;}
\label{bias}
\end{equation}
An upper bound for $R_{s,K,J}$ is explicitly provided in the next Theorem.
\begin{theorem}
\label{theorembias}Let $R_{s,K,J}$ be given by \eqref{bias}. Then, there
exist $C_{1},C_{2},C_{3}>0$ such that%
\begin{eqnarray*}
R_{s,K,J} &\leq &C_{1}B^{-rJ}+C_{2}J^{\frac{1}{2}}K^{2s-\frac{1}{2}}\exp
\left( -K^{2}\right) B^{-\left( r+2s-\frac{1}{2}\right) J}\text{ } \\
&&+C_{3}B^{\left( 1-2s\right) J}J^{\frac{1}{2}}K^{s-\frac{1}{4}%
}e^{-2K^{2}}\left( \sum_{\left\vert k\right\vert >K}\gamma _{k}\right) ^{%
\frac{1}{2}}
\end{eqnarray*}
\end{theorem}
\begin{proof}
\bigskip Using the Minkowski inequality, we have%
\begin{equation*}
\left\Vert S\left[ F\right] _{s}-S\left[ F\right] _{s,K,J}\right%
\Vert _{L^{p}\left( \mathbb{S}^{1}\right) }\leq I_{1}+I_{2}+I_{3},
\end{equation*}%
where%
\begin{align*}
&I_{1}:=\left\Vert \sum_{j=J_{0}}^{\infty }\sum_{q=1}^{Q_{j}}\beta
_{jq;s}\psi _{jq;s}-\sum_{j=J_{0}}^{J}\sum_{q=1}^{Q_{j}}\beta _{jq;s}\psi
_{jq;s}\right\Vert _{L^{2}\left( \mathbb{S}^{1}\right) };\\
&I_{2}:=\left\Vert \sum_{j=J_{0}}^{J}\sum_{q=1}^{Q_{j}}\beta _{jq;s}\psi
_{jq;s}-\sum_{j=J_{0}}^{J}\sum_{q=1}^{Q_{j}}\beta _{jq;s}\psi
_{jq;s,K}\right\Vert _{L^{2}\left( \mathbb{S}^{1}\right) };\\
&I_{3}:=\left\Vert \sum_{j=J_{0}}^{J}\sum_{q=1}^{Q_{j}}\beta _{jq;s}\psi
_{jq;sK}-\sum_{j=J_{0}}^{J}\sum_{q=1}^{Q_{j}}\beta _{jq;sK}\psi
_{jq;sK}\right\Vert _{L^{2}\left( \mathbb{S}^{1}\right) }.
\end{align*}%
Observe that while $I_{1}$ describes the bias due to the truncation of the
resolution levels belonging to $\left( J,\infty \right) $, $I_{2}$ and $%
I_{3} $ depend strictly on the choice of the cut-off frequency $K$, due to the
approximation error due to approximate Mexican needlets by the corresponding truncated ones
(the former), and Mexican needlet coefficients by the
corresponding truncated ones (the latter). According to Lemma \ref{lemmaI1}, we
get 
\begin{equation*}
I_{1}\leq C_{1,1}B^{-rJ}\text{.}
\end{equation*}%
As far as $I_{2}$ is concerned, from Lemma \ref{lemmaI2}, we obtain%
\begin{equation*}
I_{2}\leq C_{2}J^{\frac{1}{2}}K^{2s-\frac{1}{2}}\exp \left( -K^{2}\right)
B^{-\left( r+2s-\frac{1}{2}\right) J}\text{ .}
\end{equation*}%
Finally, from Lemma \ref{lemmaI3}, it holds that 
\begin{equation*}
I_{3}\leq C_{3}B^{\left( 1-2s\right) J}J^{\frac{1}{2}}K^{s-\frac{1}{4}%
}e^{-2K^{2}}\left( \sum_{\left\vert k\right\vert >K}\gamma _{k}\right) ^{%
\frac{1}{2}}\text{ ,}
\end{equation*}%
as claimed.
\end{proof}
\begin{remark}
An analogous result is obtained in Theorem 2.5 in \cite{gm2} (see also Lemma
2.3 in \cite{gm1}). In these works, the authors use a generic weight
function belonging to the Schwarz space and, moreover, the wavelets studied
are defined over a general compact manifold. For this reason, the bound
exploited in Theorem \ref{theorembias}, using explicit bounds provided by $%
w_{s}$ and by the basis $\left\{ u_{k}\right\} $, is more precise.
\end{remark}

\subsection{Adaptivity of $\widehat{F}$ for the $L^2$-risk\label{subminimaxicity}}
Merging the results achieved in the previous subsection with the ones driven
by the standard procedure in the case of nonparametric thresholding density
estimation (see for instance \cite{bkmpAoSb}), we obtain the following proof.
\begin{proof}[Proof ot the Theorem \protect\ref{theoremmain}]
Observe that, for the triangular inequality, we have%
\begin{align*}
\mathbb{E}\!\left[ \! \left\Vert \widehat{F}-F\right\Vert _{L^{2}\left( \mathbb{S}%
^{1}\right) }^{2}\!\right]\!&=\!\mathbb{E}\left[ \left\Vert \widehat{F}-S\left[ F\right] _{s,K_{n},J_{n}}+S%
\left[ F\right] _{s,K_{n},J_{n}}-S\left[ F\right] _{s}+S\left[ F\right]
_{s}-F\right\Vert _{L^{2}\left( \mathbb{S}^{1}\right) }^{2}\right]\\
&\leq E_{1}+E_{2}+E_{3}\text{ ,}
\end{align*}%
where%
\begin{eqnarray*}
E_{1} &=&\mathbb{E}\left[ \left\Vert \widehat{F}-S\left[ F\right]
_{s,K_{n},J_{n}}\right\Vert _{L^{2}\left( \mathbb{S}^{1}\right) }^{2}\right] 
\text{ ;} \\
E_{2} &=&R_{s,K_n,J_n}^{2}\text{ ;} \\
E_{3} &=&\left\Vert S\left[ F\right] _{s}-F\right\Vert _{L^{2}\left( \mathbb{%
S}^{1}\right) }^{2}\text{ .}
\end{eqnarray*}%
As far as $E_{1}$ is concerned, the bound is established analogously to
the one achieved in \cite{bkmpAoSb}, hence here it is given just a sketch
of this proof in Lemma \ref{Eproof} (see also \cite{dgm}). Indeed, we have:%
\begin{eqnarray*}
E_{1} &=&\mathbb{E}\left[ \left\Vert
\sum_{j=0}^{J_{n}}\sum_{q=1}^{Q_{j}}\left( \zeta _{jq}\left( \tau
_{n}\right) \widehat{\beta }_{jq;sK_{n}}-\beta _{jq;sK_{n}}\right) \psi
_{jq;sK_{n}}\right\Vert _{L^{2}\left( \mathbb{S}^{1}\right) }^{2}\right] \\
&\leq &\left( J_{n}+1\right) \sum_{j=0}^{J_{n}}\mathbb{E}\left[ \left\Vert
\sum_{q=1}^{Q_{j}}\left( \zeta _{jq}\left( \tau _{n}\right) \widehat{\beta }%
_{jq;sK_{n}}-\beta _{jq;sK_{n}}\right) \psi _{jq;sK_{n}}\right\Vert
_{L^{2}\left( \mathbb{S}^{1}\right) }^{2}\right] \\
&\leq &\left( J_{n}+1\right) \sum_{j=0}^{J_{n}}\sum_{q=1}^{Q_{j}}\left\Vert
\psi _{jq;sK_{n}}\right\Vert _{L^{2}\left( \mathbb{S}^{1}\right) }^{2}%
\mathbb{E}\left[ \left\vert \zeta _{jq}\left( \tau _{n}\right) \widehat{%
\beta }_{jq;sK_{n}}-\beta _{jq;sK_{n}}\right\vert ^{2}\right] \\
&\leq &\left( J_{n}+1\right) \widetilde{C_{2}}\eta
_{n}\sum_{j=0}^{J_{n}}\sum_{q=1}^{Q_{j}}\mathbb{E}\left[ \left\vert \zeta
_{jq}\left( \tau _{n}\right) \widehat{\beta }_{jq;sK_{n}}-\beta
_{jq;sK_{n}}\right\vert ^{2}\right] \\
&\leq &C_{1}J_{n}\left( E_{1,1}+E_{1,2}+E_{1,3}+E_{1,4}\right) \text{ ,}
\end{eqnarray*}%
where%
\begin{align}
E_{1,1} =&\eta _{n}\sum_{j=0}^{J_{n}}\sum_{q=1}^{Q_{j}}\mathbb{E}\left[
\left\vert \zeta _{jq}\left( \tau _{n}\right) \widehat{\beta }%
_{jq;sK_{n}}-\beta _{jq;sK_{n}}\right\vert ^{2}\mathds{1}_{\left\{
	\left\vert \widehat{\beta }_{jq;sK_{n}}\right\vert \geq \kappa \tau
	_{n}\right\} }\mathds{1}_{\left\{ \left\vert \beta _{jq;sK_{n}}\right\vert
\geq \frac{\kappa \tau _{n}}{2}\right\} }\right] ;  \label{E11} \\
E_{1,2} =&\eta _{n}\sum_{j=0}^{J_{n}}\sum_{q=1}^{Q_{j}}\mathbb{E}\left[
\left\vert \zeta _{jq}\left( \tau _{n}\right) \widehat{\beta }%
_{jq;sK_{n}}-\beta _{jq;sK_{n}}\right\vert ^{2}\mathds{1}_{\left\{
	\left\vert \widehat{\beta }_{jq;sK_{n}}\right\vert \geq \kappa \tau
	_{n}\right\} }\mathds{1}_{\left\{ \left\vert 
\beta_{jq;sK_{n}}%
\right\vert \leq \frac{\kappa \tau _{n}}{2}\right\} }\right];
\label{E12} \\
E_{1,3} =&\eta _{n}\sum_{j=0}^{J_{n}}\sum_{q=1}^{Q_{j}}\left\vert \beta
_{jq;sK_{n}}\right\vert ^{2}\mathbb{E}\left[ \mathds{1}_{\left\{ \left\vert 
\widehat{\beta }_{jq;sK_{n}}\right\vert <\kappa \tau _{n}\right\} }\mathds{1}%
_{\left\{ \left\vert \beta _{jq;sK_{n}}\right\vert \geq 2\kappa \tau
_{n}\right\} }\right]; \label{E13} \\
E_{1,4} =&\eta _{n}\sum_{j=0}^{J_{n}}\sum_{q=1}^{Q_{j}}\left\vert \beta
_{jq;sK_{n}}\right\vert ^{2}\mathbb{E}\left[ \mathds{1}_{\left\{ \left\vert 
\widehat{\beta }_{jq;sK_{n}}\right\vert <\kappa \tau _{n}\right\} }\mathds{1}%
_{\left\{ \left\vert \beta _{jq;sK_{n}}\right\vert <2\kappa \tau
_{n}\right\} }\right].  \label{E14}
\end{align}%
Heuristically, the cross-terms $E_{1,2}$ and $E_{1,3}$ are bounded by means of fast decays of the probabilistic inequalities given in Lemma \ref{probab}%
, while as far as $E_{1,1}$ and $E_{1,4}$ are concerned, their bounds will
be exploited according to the tail properties of the Besov spaces: further details
are in Lemma \ref{Eproof}. From these considerations, it follows that%
\begin{equation*}
E_{1}\leq C_{1}\left( \frac{n}{\log n}\right) ^{-\frac{2r}{2r+1}}.
\end{equation*}%
As far as $E_{2}$ is concerned, from Theorem \ref{theorembias}, it holds
that 
\begin{align*}
E_{2} =&\leq C_{1}B^{-rJ_n}+C_{2}J^{\frac{1}{2}}K_n^{2s-%
\frac{1}{2}}\exp \left( -K_n^{2}\right) B^{-\left( r+2s-\frac{1}{2}\right) J_n}  \\
& +C_{3}B^{\left( 1-2s\right) J_n}J_n^{\frac{1}{2}}K_n^{s-\frac{1}{4}%
}\exp\left({-2K_n^{2}}\right)\left( \sum_{\left\vert k\right\vert >K_n}\gamma _{k}\right) ^{%
\frac{1}{2}}
\end{align*}%
Observe that%
\begin{equation*}
B^{-2rJ_{n}}=\left( \frac{n}{\log n}\right) ^{-r}\leq \left( \frac{n}{\log n}%
\right) ^{-\frac{2r}{2r+1}}\text{ ,}
\end{equation*}%
while 
\begin{align*}
&C_{2}J_{n}^{\frac{1}{2}}K_{n}^{2s-\frac{1}{2}}\exp \left( -K_{n}^{2}\right)
B^{-\left( r+2s-\frac{1}{2}\right) J_{n}} \leq \frac{n}{\log n}^{-\frac{2r%
}{2r+1}} \\
&B^{\left( 1-2s\right) J_{n}}J_{n}^{\frac{1}{2}}K_{n}^{s-\frac{1}{4}%
}e^{-2K_{n}^{2}} \leq \frac{n}{\log n}^{-\frac{2r}{2r+1}}
\end{align*}%
Finally, we have%
\begin{equation*}
\left\Vert S\left[ F\right] _{s}-F\right\Vert _{L^{2}\left( \mathbb{S}%
^{1}\right) }^{2}\leq C_{3}\eta _{n}\text{ ,}
\end{equation*}%
whence, for $r>1$,%
\begin{equation*}
\eta _{n}\leq n^{-\frac{3}{4}}\leq \frac{n}{\log n}^{-\frac{3}{4}}\leq \frac{%
n}{\log n}^{-\frac{2r}{2r+1}}\text{ .}
\end{equation*}
\end{proof}

\section{Numerical results \label{secsimulation}}

This section presents the results of some numerical experiments.
Obviously, in the framework of finite sample situation, the asymptotic rate
given in Theorem \ref{theoremmain} has to be considered just as a prompt. In
what follows, we have built an estimator \eqref{estimatordef} using the set
to estimate $F\left( \theta \right) =\left( 2\pi \right) ^{-1}\exp \left(
\left( \theta -\pi \right) ^{2}/2\right) $ by using CRAN\ R. Some graphical
evidence can be found in Figure \ref{fig:3}. We will focus on two main
points:
\begin{itemize}
	\item the number of coefficients surviving to the thresholding procedure
	depending on $\kappa $ and $\tau _{n}$.
	\item the estimate of the $L^2$-risk function $\left\Vert \widehat{F}%
	-F\right\Vert _{L^{2}\left( \mathbb{S}^{1}\right) }$ depending on the number
	of observations $n$.
\end{itemize}
In particular, following \cite{bkmpAoSb}, we have chosen $\kappa =\kappa _{0}%
\sqrt{0.107}\sup_{\theta \in \mathbb{S}^{1}}\left\vert F\left( \theta
\right) \right\vert $, with $\kappa _{0}=0.05,0.1,0.15,0.2$, and $%
n=8000,12000$, leading to $K_{8000}=30$, $K_{12000}=36$, $J_{8000}=10$, $%
J_{12000}=11$, $t_{8000}=0.0335$ and $t_{12000}=0.028$.
The Table \ref{tab:1} counts the number of coefficients survived to
thresholding.
A qualitative analysis confirms that: (i) as $n\rightarrow \infty $, $t_{n}$
is decreasing so that the threshold is lower and more $\widehat{\beta }%
_{jq;s}$ survive to the thresholding procedure and (ii) if $\kappa _{0}$
increases, the number of surviving coefficients is smaller, expecially at
higher resolution levels.\\
The Table \ref{tab:2} describes the estimates of the $L^{2}$-risks for any
choice of $\kappa _{0}$ and any $n$. As expected, the $L^{2}$-risk$\,\ $\ is
decreasing when $n$ grows and it is increasing with respect to $\kappa _{0}$
(cfr. \cite{bkmpAoSb})
\begin{table}[tbp]
\label{tab:1}
\par
\begin{center}
\begin{tabular}{||c|c|c|c|c|c|c|c||}
\hline\hline
& \multicolumn{3}{|c|}{n=8000} & \multicolumn{3}{|c|}{n=12000} &  \\ \hline
$j\backslash \kappa _{0}$ & 0.10 & 0.15 & 0.20 & 0.10 & 0.15 & 0.20 & Tot \\ 
\hline
0 & 1 & 1 & 1 & 1 & 1 & 1 & 1 \\ \hline
1 & 1 & 1 & 1 & 1 & 1 & 1 & 1 \\ \hline
2 & 2 & 2 & 2 & 2 & 2 & 2 & 2 \\ \hline
3 & 3 & 3 & 3 & 3 & 3 & 2 & 3 \\ \hline
4 & 4 & 4 & 4 & 4 & 3 & 3 & 4 \\ \hline
5 & 5 & 4 & 4 & 5 & 5 & 5 & 5 \\ \hline
6 & 8 & 7 & 6 & 8 & 7 & 6 & 8 \\ \hline
7 & 11 & 10 & 10 & 10 & 10 & 9 & 11 \\ \hline
8 & 13 & 11 & 11 & 13 & 14 & 14 & 15 \\ \hline
9 & 19 & 14 & 12 & 20 & 21 & 18 & 21 \\ \hline
10 & 17 & 4 & 3 & 28 & 12 & 10 & 29 \\ \hline
11 & NA & NA & NA & 7 & 2 & 0 & 40 \\ \hline\hline
\end{tabular}%
\end{center}
\caption{number of mexican needlet coefficients surviving thresholding for
various values of $n$, $j$ and $\protect\kappa_0$.}
\end{table}
\begin{table}[tbp]
\label{tab:2}
\par
\begin{center}
\begin{tabular}{||c|c|c|c|c|c|c|}
\hline\hline
& \multicolumn{3}{|c|}{n=8000} & \multicolumn{3}{|c|}{n=12000} \\ \hline
$\kappa _{0}$ & 0.10 & 0.15 & 0.20 & 0.10 & 0.15 & 0.20 \\ \hline
$L^{2}$-risk & 0.481 & 0.468 & 0.451 & 0.458 & 0.432 & 0.331 \\ \hline\hline
\end{tabular}%
\end{center}
\caption{$L^2$-risk for various values of $n$ and $\protect\kappa_0$.}
\end{table}

\begin{figure}[tbp]
\centering
\includegraphics[width=\textwidth]{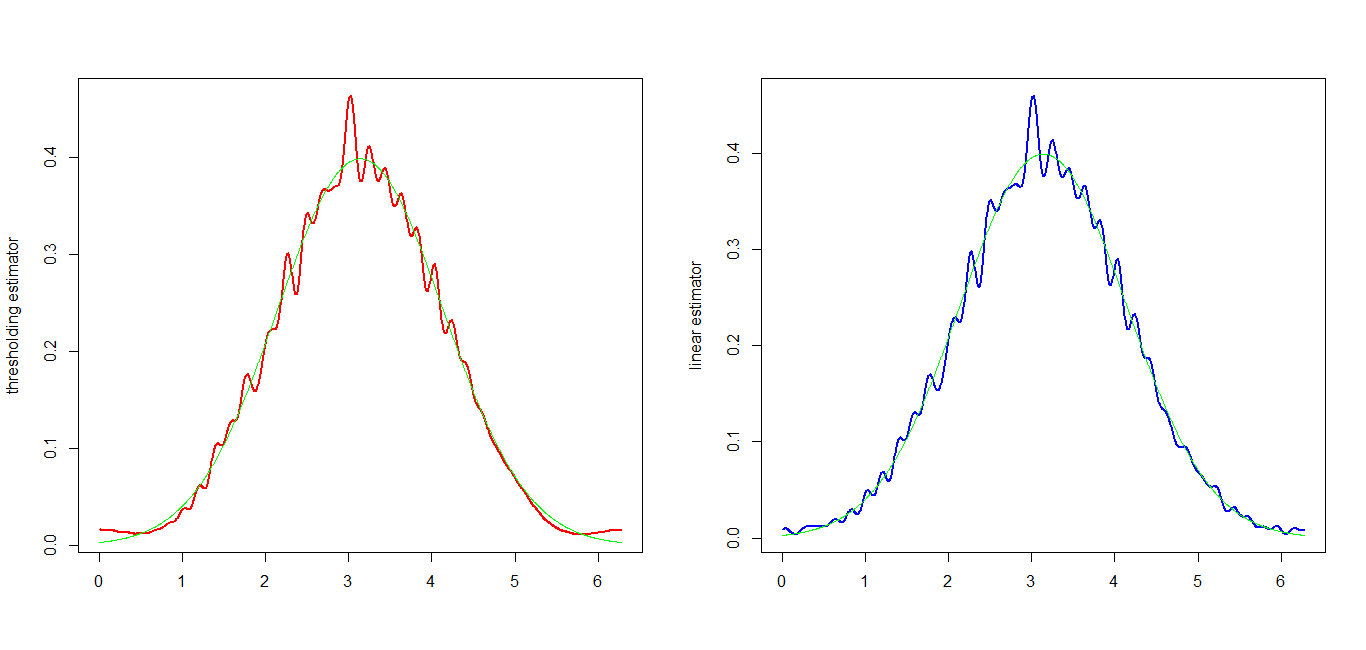}
\caption{Graphs of the thresholding estimator (on the left) and of the
linear (not-thresholded) estimator (on the right) for $n=12000$ and $s=3$.}
\label{fig:3}
\end{figure}

\section{Auxiliary results}\label{auxiliary}

This section contains all the statements and the proofs of the auxiliary results used to prove Theorem \ref{theoremmain} and Theorem \ref{theorembias}.

\subsection{Properties and Inequalities for Mexican needlets}
The first result here presented concerns the concentration property of the
Mexican needlets in the real domain.
\begin{lemma}
\label{localization}For every $\theta \in \mathbb{S}^{1},$ $s\geq 1$, there
exists $c_{s}$ such that:%
\begin{equation*}
\left\vert \psi _{jq;s}\left( \theta \right) \right\vert \!\leq \!\sqrt{\lambda
_{jq}}c_{s}B^{j}\exp \!\left(\! -\left( \frac{B^{j}\left( \theta -x_{jq}\right) 
}{2}\right) ^{2}\right) \!\left( \!
1+\left( \frac{B^{j}\left( \theta
-x_{jq}\right) }{2}\right) ^{2s}\right) .
\end{equation*}%
Furthermore, if $j\geq 0$, it holds that%
\begin{equation*}
\left\vert \psi _{jq;s}\left( \theta \right) \right\vert \leq c_{s}\eta B^{%
\frac{j}{2}}\exp \left( -\left( \frac{B^{j}\left( \theta -x_{jq}\right) }{2}%
\right) ^{2}\right) \left( 1+\left( \frac{B^{j}\left( \theta -x_{jq}\right) 
}{2}\right) ^{2s}\right).
\end{equation*}
\end{lemma}
\begin{proof}
This proof follows strictly the one for standard needlets developed in \cite%
{npw1} and the one for Mexican needlets on $\mathbb{S}^{2}$ in \cite%
{durastanti1}, see also \cite{marpecbook}. First of all, from \eqref{weight}
observe that we can define a function $W_{s}:\mathbb{R}\mapsto \mathbb{R}%
_{+} $ such that $W_{s}\left( x\right) =w_{s}\left( x^{2}\right) $. We can
therefore rewrite \eqref{needletdef} as follows:%
\begin{equation*}
\psi _{jq;s}\left( \theta \right) =\sqrt{\lambda _{jq}}\sum_{k=-\infty
}^{\infty }g_{B^{j},\theta -x_{jq}}\left( k\right),
\end{equation*}%
where%
\begin{equation*}
g_{B^{j},\theta ,x_{jq}}\left( u\right) :=W_{s}\left( uB^{-j}\right) \exp
\left( iu\left( \theta -x_{jq}\right) \right) .
\end{equation*}%
For the Poisson summation formula (see for instance \cite{npw1, durastanti1,
marpecbook}), we have that%
\begin{equation*}
\sum_{k=-\infty }^{\infty }g_{B^{j},\theta ,x_{jq}}\left( k\right)
=\sum_{\nu =-\infty }^{\infty }\mathcal{F}\left[ g_{B^{j},\theta ,x_{jq}}%
\right] \left( 2\pi \nu \right) ,
\end{equation*}%
where the symbol $\mathcal{F}\left[ g\right] $ denotes the Fourier transform
of $g$. In our case, we have 
\begin{equation*}
\mathcal{F}\left[ g_{B^{j},\theta x_{jq}}\right] \left( \omega \right) =%
\mathcal{F}\left[ W_{s}\left( uB^{-j}\right) \right] \ast \mathcal{F}\left[
\exp \left( iu\left( \theta -x_{jq}\right) \right) \right] ,
\end{equation*}%
where the symbol $\ast $ denotes the convolution product. Standard
calculations lead to%
\begin{align*}
&\mathcal{F}\left[ W_{s}\left( uB^{-j}\right) \right] =\frac{\left(
-1\right) ^{s}}{\sqrt{2}B^{-j}}H_{2s}\left( \frac{\omega }{2B^{-j}}\right)
\exp \left( -\left( \frac{\omega }{2B^{-j}}\right) ^{2}\right) ; \\
&\mathcal{F}\left[ \exp \left( iu\left( \theta -x_{jq}\right) \right) \right]
=\sqrt{2\pi }\delta \left( \left( \theta -x_{jq}\right) -\omega \right) .
\end{align*}%
Hence we get
\begin{equation*}
\mathcal{F}\left[ g_{B^{j},\theta ,x_{jq}}\right] \left( \omega \right) = %
\frac{\left( -1\right) ^{s}\sqrt{\pi }}{B^{-j}}H_{2s}\left( \frac{\left(
\theta -x_{jq}\right) -\omega }{2B^{-j}}\right) \exp \left( -\left( \frac{%
\left( \theta -x_{jq}\right) -\omega }{2B^{-j}}\right) ^{2}\right) .
\end{equation*}%
Following Proposition 2 in \cite{durastanti1}, we have that 
\begin{equation*}
\sum_{\nu = -\infty }^{\infty }\mathcal{F}\left[ \widehat{g}_{B^{j},\theta
-x_{jq}}\right] \left( 2\pi \nu \right) \leq  C_{2s}B^{j}\exp \left( -\left( 
\frac{B^{j}\left( \theta -x_{jq}\right) }{2}\right) ^{2}\right)  H_{2s}\left( 
\frac{B^{j}\left( \theta -x_{jq}\right) }{2}\right) .
\end{equation*}%
It can be easily proved that
\begin{equation*}
H_{2s}\left( \frac{B^{j}\left( \theta
-x_{jq}\right) }{2}\right) 
\approx \left( 1+\left( 
\frac{B^{j}\left( \theta -x_{jq}\right) }{2}\right) ^{2s}\right),
\end{equation*}%
see for instance \cite{durastanti2}. Straightforward calculations lead to the claimed result.
\end{proof}
The next results will be pivotal to truncate negative resolution levels and
the frequencies $k:\left\vert k\right\vert >K$ in the summation formula.
\begin{lemma}
\label{ecchecaspita}Let $w_{s}:\mathbb{R}\mapsto \mathbb{R}_{+}$ be given by
\eqref{weight}. Let $J_{0}\in \mathbb{N}$. Hence, for $t>0$, it holds that%
\begin{equation}
\sum_{j=-\infty }^{-J_{0}}\left\vert w_{s}\left( tB^{-2j}\right) \right\vert
^{2}=\frac{\chi _{s,B,J_{0}}\left( t\right) }{2\log B}\left( 1\pm O\left(
\left\vert B-1\right\vert ^{2}\log \left\vert B-1\right\vert \right) \right),  \label{supweight}
\end{equation}%
where 
\begin{equation*}
\chi _{s,B,J_{0}}\left( t\right) :=2^{-2s}\Gamma \left( 2s,2tB^{\frac{J_{0}}{%
\log B}}\right).
\end{equation*}%
Furthermore, it holds that%
\begin{equation}
\sum_{j=J_{0}}^{\infty }\left\vert w_{s}\left( tB^{-2j}\right) \right\vert
^{2}=\frac{\phi _{s,B,J_{0}}\left( t\right) }{2\log B}\left( 1\pm O\left(
\left\vert B-1\right\vert ^{2}\log \left\vert B-1\right\vert \right) \right) ,\label{infweight}
\end{equation}%
where 
\begin{equation*}
\phi _{s,B,J_{0}}\left( t\right) :=2^{-2s}\gamma \left( 2s,2tB^{\frac{J_{0}}{%
\log B}}\right).
\end{equation*}
\end{lemma}
\begin{proof}
Let us start by proving \eqref{supweight}. First of all, observe that the
following identity holds:%
\begin{equation*}
\int_{B^{\frac{J_{0}}{\log B}}}^{\infty }\left\vert w_{s}\left( tx\right)
\right\vert ^{2}\frac{dx}{x}=2^{-2s}\Gamma \left( 2s,2tB^{\frac{J_{0}}{\log B%
}}\right) =\chi _{s,B,J_{0}}\left( t\right).
\end{equation*}%
Applying an analogous procedure to the one adopted in Lemma 7.6 in \cite{gm0}%
, define the function $G_{s}:\mathbb{R}\mapsto \mathbb{R}_{+}$ by%
\begin{equation*}
G_{s}\left( u\right) :=\left\vert w_{s}\left( e^{u}\right) \right\vert
^{2}=e^{-2e^{u}\left( 1-sue^{-u}\right) }.
\end{equation*}%
Let $j^{\prime }=-j$, and fix $t=e^{v}$, $v>0$; on one hand we get%
\begin{equation*}
\sum_{j^{\prime }=J_{0}}^{\infty }\left\vert w_{s}\left( tB^{2j^{\prime
}}\right) \right\vert ^{2}=\sum_{j^{\prime }=J_{0}}^{\infty }\left\vert
w_{s}\left( e^{dj^{\prime }+v}\right) \right\vert ^{2}=\sum_{j^{\prime
}=J_{0}}^{\infty }G_{s}\left( dj^{\prime }+v\right),
\end{equation*}%
$d=2\log B$. On the other hand, for $u=\log x$, we obtain 
\begin{equation*}
\int_{B^{\frac{J_{0}}{\log B}}}^{\infty }\left\vert w_{s}\left( tx\right)
\right\vert ^{2}\frac{dx}{x}=\int_{J_{0}}^{\infty }\left\vert w_{s}\left(
e^{u+v}\right) \right\vert ^{2}du=\int_{J_{0}}^{\infty }G_{s}\left(
u+v\right) du.
\end{equation*}%
As in Lemma 7.6 in \cite{gm0}, note that $d\sum_{j^{\prime }=J_{0}}^{\infty
}G_{s}\left( dj^{\prime }+v\right) $ is a Riemann sum for $\int_{B^{\frac{%
J_{0}}{\log B}}}^{\infty }G_{s}\left( u+v\right) du$. Moreover, because $%
\sum_{j^{\prime }=J_{0}}^{\infty }G_{s}\left( dj^{\prime }+v\right) $ is
periodic with period $d$, it is sufficient to estimate this sum just for $%
0<v<d$. Now observe that, for $J=J_{0}+\Delta J$, $\Delta J>0$, we get
\begin{align*}
\left\vert d\sum_{j^{\prime }=J_{0}}^{\infty }G_{s}\left( dj^{\prime
}+v\right) -\chi _{s,B,J_{0}}\left( t\right) \right\vert = &\left\vert
d\sum_{j^{\prime }=J_{0}}^{\infty }G_{s}\left( dj^{\prime }+v\right)
-\int_{J_{0}}^{\infty }G_{s}\left( u+v\right) du\right\vert \\
\leq & \left\vert d \!\sum_{j^{\prime }=J_{0}}^{J} \! G_{s} \! \left( dj^{\prime
}+v\right)\! -\!\int_{J_{0}}^{Jd+\frac{d}{2}}\!G_{s}\!
\left( u+v\right)\!
du\right\vert \\  + & d\sum_{j^{\prime }>J}G_{s}\left( dj^{\prime }+v\right)\!
+\!\int_{Jd+\frac{d}{2}}^{\infty }G_{s}\left( u+v\right) du.
\end{align*}%
Using the midpoint rule, (see again Lemma 7.6 in \cite{gm0}), we obtain
\begin{equation*}
\left\vert d\sum_{j^{\prime }=J_{0}}^{J}G_{s}\left( dj^{\prime }+v\right)
-\int_{J_{0}}^{Jd+\frac{d}{2}}G_{s}\left( u+v\right) du\right\vert \leq 
\frac{1}{24}\left\Vert G^{\prime \prime }\right\Vert _{\infty }\left(
J-J_{0}\right) d^{3}.
\end{equation*}%
On the other hand, observe that, for $r>0$, 
\begin{equation*}
\frac{d}{dr}\left( 1-sre^{-r}\right) =se^{-r}\left( r-1\right),
\end{equation*}%
so that $\left( 1-sre^{-r}\right) $ is monotonically decreasing for $r\in %
\left[ 0,1\right) $, it attains its minimum for $r=1$ and then it is
monotonically increasing for $r\in \left( 1,\infty \right) $. Observing that $%
\left( 1-sre^{-r}\right) _{r=0}=1$ and $\lim_{r\rightarrow \infty }\left(
1-sre^{-r}\right) =1$ yields to 
\begin{equation*}
G_{s}\left( r\right) \leq e^{-2e^{r}\left( 1-sre^{-r}\right) }\leq
e^{-2e^{r}}.
\end{equation*}%
Consequently, for $j^{\prime \prime }=\exp dj^{\prime}$, we obtain 
\begin{equation*}
d\sum_{j^{\prime }>J}G_{s}\left( dj^{\prime }+v\right) \leq \sum_{j^{\prime
}>J}e^{-2e^{dj^{\prime }+v}}=\sum_{j^{\prime \prime }>\exp
dJ}e^{-2e^{v}j^{\prime \prime }}=\frac{e^{2e^{v}\left( 1-e^{dJ}\right) }}{%
e^{2e^{v}}-1},
\end{equation*}%
which, for $y=2e^{u+v}$, leads to 
\begin{eqnarray*}
\int_{Jd+\frac{d}{2}}^{\infty }G_{s}\left( u+v\right) du &=&\int_{Jd+\frac{d%
}{2}}^{\infty }e^{-2e^{u+v}}du\\
&=&\int_{2e^{Jd+\frac{d}{2}+v}}^{\infty }e^{-y}%
\frac{dy}{y} \\
&\leq &2e^{-\left( Jd+\frac{d}{2}+v\right) }e^{-2e^{Jd+\frac{d}{2}+v}}.
\end{eqnarray*}%
Therefore, there exists a constant $C>0$ so that 
\begin{eqnarray*}
\left\vert \sum_{j^{\prime }=J_{0}}^{\infty }G_{s}\left( dj^{\prime
}+v\right) -\frac{\chi _{s,B,J_{0}}\left( t\right) }{d}\right\vert &\leq &%
\frac{1}{24}\left\Vert G^{\prime \prime }\right\Vert _{\infty }\Delta Jd^{2}+%
\frac{e^{2e^{v}\left( 1-e^{dJ}\right) }}{e^{2e^{v}}-1}\\ &&+2e^{-\left( Jd+\frac{d%
}{2}+v\right) }e^{-2e^{Jd+\frac{d}{2}+v}} \\
&\leq &C\left( \Delta Jd^{2}+e^{-2e^{Jd}}\right) \leq C^{\prime }\Delta
Jd^{2}.
\end{eqnarray*}%
According again to \cite{gm0}, we choose $\Delta J\in \left( \log \left(
1/d\right) /d,2\log \left( 1/d\right) /d\right) $ so that%
\begin{equation*}
\left\vert \sum_{j=-\infty }^{-J_{0}}\left\vert w_{s}\left( tB^{-2j}\right)
\right\vert ^{2}-\frac{\chi _{s,B,J_{0}}\left( t\right) }{d}\right\vert \leq
C^{\prime }\left( 2d\log \left( \frac{1}{d}\right) \right) .
\end{equation*}%
It follows that 
\begin{equation*}
\left \vert \left(\frac{\chi _{s,B,J_{0}}\left(
	t\right) }{d}\right)^{-1}\sum_{j=-\infty }^{-J_{0}}\left\vert w_{s}\left(
tB^{-2j}\right) \right\vert ^{2} -1 \right \vert \leq 2\frac{C^{\prime }}{\chi _{s,B,J_{0}}\left(
t\right) }d^{2}\log \left( \frac{1}{d}\right). 
\end{equation*}
Finally, because $%
d=2\log B$ and $\lim_{B\rightarrow 1^{+}}\log B/\left( B-1\right) =1$, the
proof is complete. The proof of \eqref{infweight} is totally analogous and,
therefore, omitted.
\end{proof}
\begin{lemma}
\label{ecchecaspitak}Let $w_{s}:\mathbb{R}\mapsto \mathbb{R}_{+}$ be given
by \eqref{weight}. Then we have%
\begin{equation*}
\sum_{\left\vert k\right\vert >K}w_{s}^{2}\left( \left( kB^{-j}\right)
^{2}\right) \leq 2^{-\left( 2s+\frac{1}{2}\right) }B^{j}\Gamma \left( 2s+%
\frac{1}{2},2K^{2}B^{-2j}\right).
\end{equation*}
\end{lemma}
\begin{proof}
Observe that%
\begin{eqnarray*}
\sum_{\left\vert k\right\vert >K}w_{s}^{2}\left( \left( kB^{-j}\right)
^{2}\right) &=&\sum_{\left\vert k\right\vert >K}\left( B^{-j}k\right)
^{4s}\exp \left( -2\left( B^{-j}k\right) ^{2}\right) \\
&\leq &2\int_{K}^{\infty }\left( B^{-j}x\right) ^{4s}\exp \left( -2\left(
B^{-j}x\right) ^{2}\right) dx \\
&\leq &2^{-\left( 2s+\frac{1}{2}\right) }B^{j}\left(
\int_{2K^{2}B^{-2j}}^{\infty }u^{2s-\frac{1}{2}}\exp \left( -u\right)
du\right) \\
&\leq &2^{-\left( 2s+\frac{1}{2}\right) }B^{j}\Gamma \left( 2s+\frac{1}{2}%
,2K^{2}B^{2-j}\right) ,
\end{eqnarray*}%
as claimed.
\end{proof}
\begin{corollary}
\label{ecchecaspitakcor}Let $\omega ,J>0$; for $x$ sufficiently large, it
holds that 
\begin{equation*}
\sum_{j=J_{0}}^{J}B^{-\omega j}\Gamma \left( S+1,xB^{-2j}\right) \leq
C_{S,\alpha }x^{S}e^{-x}B^{-\left( \omega +2S\right) J}.
\end{equation*}
\end{corollary}
\begin{proof}
For $x$ sufficiently large, the following limit holds%
\begin{equation*}
\lim_{x\rightarrow \infty }\frac{\Gamma \left( S+1,xB^{-2j}\right) }{\left(
xB^{-2j^{\prime }}\right) ^{S}e^{-xB^{-2j}}}=1,
\end{equation*}%
see for instance \cite{abramovitzstegun} Formula 6.5.32, pag. 263.
Therefore, we have:%
\begin{equation*}
\sum_{j=0}^{J}B^{-\left( \omega +2S\right) j}x^{S}e^{-xB^{-2j}}\leq
x^{S}e^{-x}\sum_{j=0}^{J}B^{-\left( \omega +2S\right) j}\leq C_{S,\omega
}x^{S}e^{-x}B^{-\left( \omega +2S\right) J}.
\end{equation*}
\end{proof}
The next result concerns the behaviour of the sums of the powers of the
weights $\lambda _{jq}$.
\begin{lemma}
\label{lemmalambda}Let $Q_{j}$ and $\lambda _{jq}$ be so that Theorem \ref%
{nearlytight} holds. For any $j$, it holds that 
\begin{equation*}
\sum_{q=1}^{Q_{j}}\lambda _{jq}\approx 1.
\end{equation*}%
Furthermore, let $p>1$ and $j>0$. It holds that 
\begin{equation}
\sum_{q=1}^{Q_{j}}\lambda _{jq}^{p}\leq \eta ^{p-1}B^{j\left( 1-p\right) }.  \label{lambasomma}
\end{equation}
\end{lemma}
\begin{proof}
The first inequality follows directly the conditions in Theorem \ref%
{nearlytight}. On the other hand, for $j>0$, it is immediate to see that%
\begin{equation*}
\sum_{q=1}^{Q_{j}}\lambda _{jq}^{p}\leq \eta
^{p}\sum_{q=1}^{Q_{j}}B^{-jp}\leq \eta ^{p}Q_{j}B^{-jp}\leq \eta
^{p-1}B^{j\left( 1-p\right) }.
\end{equation*}
\end{proof}
The next Lemma establishes explicit upper bounds for the sums with respect
to $q$ of differences between Mexican standard and truncated coefficients
and for the $L^{2}$-norms of the sums with respect to $q$ of Mexican
standard and truncated needlets.
\begin{lemma}
\label{formulalemma}For $j>0$, it holds that%
\begin{align}
&\sum_{q=1}^{Q_{j}}\left\vert \beta _{jq;sK}-\beta _{jq;s}\right\vert
^{2}\leq B^{j}\Gamma \left( 2s+\frac{1}{2},2K^{2}B^{-2j}\right)
\sum_{\left\vert k\right\vert >K}\gamma _{k}\text{ ;}  \label{formulalemma1}\\
&\left\Vert \left( \psi _{jq;sK}-\psi _{jq;s}\right) \right\Vert
_{L^{2}\left( \mathbb{S}^{1}\right) }^{2}\leq 2^{-\left( 2s+\frac{1}{2}%
\right) }\eta \Gamma \left( 2s+\frac{1}{2},2K^{2}B^{-2j}\right) \text{ .}
\label{formulalemma2}
\end{align}
\end{lemma}
\begin{proof}
For the H\"{o}lder inequality, it holds that%
\begin{eqnarray*}
\sum_{q=1}^{Q_{j}}\left\vert \beta _{jq;sK}-\beta _{jq;s}\right\vert ^{2}
&=&\sum_{q=1}^{Q_{j}}\lambda _{jq}\left\vert \sum_{\left\vert k\right\vert
>K}w_{s}\left( \left( kB^{-j}\right) ^{2}\right) a_{k}u_{k}\left( \xi
_{jq}\right) \right\vert ^{2} \\
&\leq &\sum_{q=1}^{Q_{j}}\lambda _{jq}\left( \sum_{\left\vert k\right\vert
>K}w_{s}\left( \left( kB^{-j}\right) ^{2}\right) \left\vert a_{k}\right\vert
\left\vert u_{k}\left( \xi _{jq}\right) \right\vert \right) ^{2} \\
&\leq &\sum_{q=1}^{Q_{j}}\lambda _{jq}\sum_{\left\vert k\right\vert
>K}w_{s}^{2}\left( \left( kB^{-j}\right) ^{2}\right) \sum_{\left\vert
k\right\vert >K}\left\vert a_{k}\right\vert ^{2}\left\vert u_{k}\left( \xi
_{jq}\right) \right\vert ^{2} \\
&\leq &\sum_{q=1}^{Q_{j}}\lambda _{jq}\sum_{\left\vert k\right\vert
>K}w_{s}^{2}\left( \left( kB^{-j}\right) ^{2}\right) \sum_{\left\vert
k\right\vert >K}\gamma _{k}.
\end{eqnarray*}%
Using Lemma \ref{ecchecaspitak} and Lemma \ref{lemmalambda} leads to the claimed result. As far as \eqref{formulalemma2} is concerned, we get%
\begin{eqnarray*}
\left( \psi _{jq;sK}\left( \theta \right) -\psi _{jq;s}\left( \theta \right)
\right) ^{2} &=&\lambda _{jq}\left( \sum_{\left\vert k\right\vert >K}w_{s}\left( \left(
kB^{-j}\right) ^{2}\right) \overline{u_{k}\left( \xi _{jq}\right) }%
u_{k}\left( \theta \right) \right) ^{2} \\
&=&\int_{\mathbb{S}^{1}}\lambda _{jq}\left( \sum_{\left\vert
k_{1}\right\vert ,\left\vert k_{2}\right\vert >K}w_{s}^{2}\left( \left(
kB^{-j}\right) ^{2}\right) \right) .
\end{eqnarray*}%
Using the orthogonality of $\left\{ u_{k}\right\} $ and Lemma \ref%
{ecchecaspitak} yields to%
\begin{eqnarray*}
\left\Vert \left( \psi _{jq;sK}-\psi _{jq;s}\right) \right\Vert
_{L^{2}\left( \mathbb{S}^{1}\right) }^{2} &=&\int_{\mathbb{S}^{1}}\lambda
_{jq}\left( \sum_{\left\vert k_{1}\right\vert ,\left\vert k_{2}\right\vert
>K}w_{s}^{2}\left( \left( kB^{-j}\right) ^{2}\right) \right) \\
&\leq &2^{-\left( 2s+\frac{1}{2}\right) }\eta \Gamma \left( 2s+\frac{1}{2}%
,2K^{2}B^{-2j}\right) ,
\end{eqnarray*}%
as claimed.
\end{proof}

\subsection{Ancillary results related to Theorem \ref{theorembias}}
The lemmas here proved describe the behaviour of  $I_{1}$, $I_{2}$ and $I_{3}$. Hence, they are pivotal to study the bias $R_{s,K,J,\eta}$.
\begin{lemma}
\label{lemmaI1}Let $I_{1}$ be given by%
\begin{equation*}
I_{1}:=\left\Vert \sum_{j>J}\sum_{q=1}^{Q_{j}}\beta _{jq;s}\psi
_{jq;s}\right\Vert _{L^{2}\left( \mathbb{S}^{1}\right) }.
\end{equation*}%
Then, there exists $C_{1}>0$ such that 
\begin{equation*}
I_{1}\leq C_{1}B^{-rJ}.
\end{equation*}
\end{lemma}
\begin{proof}
Note preliminarily that%
\begin{equation*}
I_{1}\leq\sum_{j>J}\left\Vert \sum_{q=1}^{Q_{j}}\beta _{jq;s}\psi
_{jq;s}\right\Vert _{L^{2}\left( \mathbb{S}^{1}\right) }.
\end{equation*}%
Observe that, for the H\"{o}lder inequality (see also \cite{bkmpAoSb}), we
get%
\begin{eqnarray*}
\left( \sum_{q=1}^{Q_{j}}\left\vert \beta _{jq;s}\psi _{jq;s}\left( \theta
\right) \right\vert \right) ^{2} &=&\left( \sum_{q=1}^{Q_{j}}\left\vert
\beta _{jq;s}\right\vert \left\vert \psi _{jq;s}\left( \theta \right)
\right\vert ^{\frac{1}{2}}\left\vert \psi _{jq;s}\left( \theta \right)
\right\vert ^{\frac{1}{2}}\right) ^{2} \\
&\leq &\left( \sum_{q=1}^{Q_{j}}\left\vert \beta _{jq;s}\right\vert
^{2}\left\vert \psi _{jq;s}\left( \theta \right) \right\vert \right) \left(
\sum_{q=1}^{Q_{j}}\left\vert \psi _{jq;s}\left( \theta \right) \right\vert
\right) \\
&\leq &C\eta ^{\frac{1}{2}}B^{\frac{j}{2}}\sum_{q=1}^{Q_{j}}\left\vert \beta
_{jq;s}\right\vert ^{2}\left\vert \psi _{jq;s}\left( \theta \right)
\right\vert.
\end{eqnarray*}%
For $C>0$ and using \eqref{besovineq}, it follows that 
\begin{eqnarray*}
\left\Vert \sum_{q=1}^{Q_{j}}\beta _{jq;s}\psi _{jq;s}\right\Vert
_{L^{2}\left( \mathbb{S}^{1}\right) }^{2} &\leq &C\eta ^{\frac{1}{2}}B^{%
\frac{j}{2}}\sum_{q=1}^{Q_{j}}\left\vert \beta _{jq;s}\right\vert
^{2}\left\Vert \psi _{jq;s}\right\Vert _{L^{1}\left( \mathbb{S}^{1}\right) }
\\
&\leq &C\sum_{q=1}^{Q_{j}}\eta \left\vert \beta _{jq;s}\right\vert ^{2}\leq
CB^{-2rj}.
\end{eqnarray*}%
Hence, we obtain%
\begin{equation*}
I_{1,1} =\sum_{j>J}\left\Vert \sum_{q=1}^{Q_{j}}\beta _{jq;s}\psi
_{jq;s}\right\Vert _{L^{2}\left( \mathbb{S}^{1}\right) } \leq C_{1,2}B^{-rJ}\text{ .}
\end{equation*}
\end{proof}
\begin{lemma}
\label{lemmaI2}Let $I_{2}$ be given by%
\begin{equation*}
I_{2}:=\left\Vert \sum_{j=J_{0}}^{J}\sum_{q=1}^{Q_{j}}\beta _{jq;s}\left(
\psi _{jq;s}-\psi _{jq;sK}\right) \right\Vert _{L^{2}\left( \mathbb{S}%
^{1}\right) } .
\end{equation*}%
Then, there exists $C_{2}>0$ such that%
\begin{equation*}
I_{2}\leq C_{2}J^{\frac{1}{2}}K^{2s-\frac{1}{2}}\exp \left( -K^{2}\right)
B^{-\left( r+2s-\frac{1}{2}\right) J}.
\end{equation*}
\end{lemma}
\begin{proof}
First of all, observe%
\begin{align*}
&\left\Vert \sum_{j=J_{0}}^{J}\sum_{q=1}^{Q_{j}}\beta _{jq;s}\left( \psi
_{jq;s}-\psi _{jq;sK}\right) \right\Vert _{L^{2}\left( \mathbb{S}^{1}\right)
}^{2} \\ &\quad\quad\quad\quad\quad\quad \leq \left( J-J_{0}+1\right) \sum_{j=J_{0}}^{J}\left\Vert
\sum_{q=1}^{Q_{j}}\beta _{jq;s}\left( \psi _{jq;s}-\psi _{jq;sK}\right)
\right\Vert _{L^{2}\left( \mathbb{S}^{1}\right) }^{2};
\end{align*}%
using the H\"{o}lder inequality, we have%
\begin{eqnarray*}
\left( \sum_{q=1}^{Q_{j}}\left\vert \beta _{jq;s}\left( \psi _{jq;s}\left(
\theta \right) -\psi _{jq;sK}\left( \theta \right) \right) \right\vert
\right) ^{2}
& \leq &\left( \sum_{q=1}^{Q_{j}}\left( \psi _{jq;s}\left( \theta
\right) -\psi _{jq;sK}\left( \theta \right) \right) ^{2}\right)\\
&& \left( \sum_{q=1}^{Q_{j}}\left\vert \beta _{jq;s}\right\vert
^{2}\right) ,
\end{eqnarray*}%
so that%
\begin{equation*}
\left\Vert \sum_{q=1}^{Q_{j}}\beta _{jq;s}\left( \psi _{jq;s}\!-\!\psi
_{jq;sK}\right) \right\Vert _{L^{2}\left( \mathbb{S}^{1}\right) }^{2}\!\!\!\!\leq\!
\left( \sum_{q=1}^{Q_{j}}\left\vert \beta _{jq;s}\right\vert ^{2}\right)\!
\sum_{q=1}^{Q_{j}}\!\left\Vert \left( \psi _{jq;sK}\!-\!\psi _{jq;s}\right)
\right\Vert _{L^{2}\left( \mathbb{S}^{1}\right) }^{2}\! .
\end{equation*}%
Using \eqref{formulalemma2} in Lemma \ref{formulalemma} and \eqref{besovineq}, we obtain%
\begin{equation*}
\left\Vert \sum_{q=1}^{Q_{j}}\beta _{jq;s}\left( \psi _{jq;s}-\psi
_{jq;sK}\right) \right\Vert _{L^{2}\left( \mathbb{S}^{1}\right) }^{2}\leq
CB^{-2rj}\Gamma \left( 2s+\frac{1}{2},2K^{2}B^{-2j}\right)
\end{equation*}%
so that, from the Corollary \ref{ecchecaspitakcor}, it follows%
\begin{equation*}
\sum_{j=J_{0}}^{J}B^{-2rj}\Gamma \left( 2s+\frac{1}{2},2K^{2}B^{-2j}\right)
\leq CK^{4s-1}\exp \left( -2K^{2}\right) B^{-\left( 2r+4s-1\right) J}.
\end{equation*}%
Hence, we get  
\begin{equation*}
I_{2}\leq \text{ }C_{2}J^{\frac{1}{2}}K^{2s-\frac{1}{2}}\exp \left(
-K^{2}\right) B^{-\left( r+2s-\frac{1}{2}\right) J},
\end{equation*}%
as claimed.
\end{proof}
\begin{lemma}
\label{lemmaI3}Let $I_{3}$ be given by%
\begin{equation*}
I_{3}:=\left\Vert \sum_{j=J_{0}}^{J}\sum_{q=1}^{Q_{j}}\left( \beta
_{jq;s}-\beta _{jq;sK}\right) \psi _{jq;sK}\right\Vert _{L^{2}\left( \mathbb{%
S}^{1}\right) }.
\end{equation*}%
Then, there exists $C_{3}>0$ such that 
\begin{equation*}
I_{3}\leq C_{3}B^{\left( \frac{3}{2}-2s\right) J}J^{\frac{1}{2}}K^{2s-\frac{1%
}{2}}e^{-K^{2}}\left( \sum_{\left\vert k\right\vert >K}\gamma _{k}\right) ^{%
\frac{1}{2}}.
\end{equation*}
\end{lemma}
\begin{proof}
Observe that%
\begin{align*}
&\left\Vert \sum_{j=J_{0}}^{J}\sum_{q=1}^{Q_{j}}\left( \beta _{jq;s}-\beta
_{jq;sK}\right) \psi _{jq;sK}\right\Vert _{L^{2}\left( \mathbb{S}^{1}\right)
}^{2}\\
&\quad\quad\quad\quad\quad\quad \leq \left( J-J_{0}+1\right) \sum_{j=J_{0}}^{J}\left\Vert
\sum_{q=1}^{Q_{j}}\left( \beta _{jq;s}-\beta _{jq;sK}\right) \psi
_{jq;sK}\right\Vert _{L^{2}\left( \mathbb{S}^{1}\right) }^{2}.
\end{align*}%
H\"{o}lder inequality leads to%
\begin{eqnarray*}
\left( \sum_{q=1}^{Q_{j}}\left\vert \left( \beta _{jq;s}-\beta
_{jq;sK}\right) \psi _{jq;sK}\left( \theta \right) \right\vert \right)
^{2} &\leq &\left( \sum_{q=1}^{Q_{j}}\left( \beta _{jq;s}-\beta _{jq;sK}\right)
^{2}\left\vert \psi _{jq;sK}\left( \theta \right) \right\vert \right) \\ && \left(
\sum_{q=1}^{Q_{j}}\left\vert \psi _{jq;sK}\left( \theta \right) \right\vert
\right).
\end{eqnarray*}%
We therefore obtain 
\begin{equation*}
\sum_{q=1}^{Q_{j}}\left\vert \psi _{jq;s,K}\left( \theta \right) \right\vert
\leq \sum_{q=1}^{Q_{j}}\left\vert \psi _{jq;s}\left( \theta \right)
\right\vert \leq CB^{\frac{j}{2}};
\end{equation*}%
so that%
\begin{eqnarray*}
\left\Vert \sum_{q=1}^{Q_{j}}\left( \beta _{jq;s}-\beta _{jq;sK}\right) \psi
_{jq;sK}\right\Vert _{L^{2}\left( \mathbb{S}^{1}\right) }^{2}\!\!\!\!\! &\leq &CB^{%
\frac{j}{2}}\sum_{q=1}^{Q_{j}}\left( \beta _{jq;s}-\beta _{jq;sK}\right)
^{2}\left\Vert \psi _{jq;sK}\right\Vert _{L^{1}\left( \mathbb{S}^{1}\right) }
\\
&\leq &CB^{j}\Gamma \left( 2s+\frac{1}{2},2K^{2}B^{-2j}\right)
\sum_{\left\vert k\right\vert >K}\gamma _{k}.
\end{eqnarray*}%
Using Corollary \ref{ecchecaspitakcor} leads to%
\begin{equation*}
\sum_{j=J_{0}}^{J}B^{j}\Gamma \left( 2s+\frac{1}{2},2K^{2}B^{-2j}\right)
\leq C\left( 2K^{2}\right) ^{2s-\frac{1}{2}}\exp \left( -2K^{2}\right)
B^{-2\left( 2s-1\right) J}.
\end{equation*}%
Hence, we get
\begin{equation*}
I_{3}\leq C_{3}B^{\left( 1-2s\right) J}J^{\frac{1}{2}}K^{s-\frac{1}{4}%
}e^{-2K^{2}}\left( \sum_{\left\vert k\right\vert >K}\gamma _{k}\right) ^{%
\frac{1}{2}}\text{ ,}
\end{equation*}%
as claimed.
\end{proof}

\subsection{Ancillary results related to Theorem \protect\ref{theoremmain}}
In this subsection we will summon auxiliary results connected to the proof Theorem \ref{theoremmain}.
\begin{lemma}
\label{Eproof}Let $E_{1,1}$, $E_{1,2}$, $E_{1,3}$ and $E_{1,4}$ be given
respectively by \eqref{E11}, \eqref{E12}, \eqref{E13} and \eqref{E14}. Then,
there exists $C_{E}>0$ such that%
\begin{equation*}
E_{1,1}+E_{1,2}+E_{1,3}+E_{1,4}\leq C_{E}\left( \frac{n}{\log n}\right) ^{-%
\frac{2r}{2r+1}}.
\end{equation*}
\end{lemma}
\begin{proof}
Observe that%
\begin{equation*}
E_{1,1}\leq C_{1}\eta _{n}n^{-1}\left( \sum_{j=0}^{J_n}\sum_{q=1}^{Q_{j}}%
\mathds{1}_{\left\{ \left\vert \beta _{jq;sK_{n}}\right\vert \geq \frac{%
\kappa \tau _{n}}{2}\right\} }\right),
\end{equation*}%
Splitting the sum into two parts by means of the so-called optimal bandwidth selection, given by $J_{1,n}:B^{J_{1,n}}=\left( n/\log n\right) ^{%
\frac{1}{2r+1}}$, and using \eqref{Eineq} yields to 
\begin{equation*}
\eta _{n}\sum_{j=0}^{J_{1,n}}\sum_{q=1}^{Q_{j}}\mathds{1}_{\left\{
\left\vert \beta _{jq;sK_{n}}\right\vert \geq \frac{\kappa \tau _{n}}{2}%
\right\} }\leq CB^{J_{1,n}}\leq C\left( n/\log n\right) ^{\frac{1}{2r+1}},
\end{equation*}%
and, on the other hand,  
\begin{eqnarray*}
\eta _{n}\sum_{j=J_{1,n}}^{J_{n}}\mathds{1}_{\left\{ \left\vert \beta
_{jq;sK_{n}}\right\vert \geq \frac{\kappa \tau _{n}}{2}\right\} } &\leq
&C\eta _{n}\sum_{j=J_{1,n}}^{J_{n}}\sum_{q=1}^{Q_{j}}\left\vert \beta
_{jq;sK_{n}}\right\vert ^{2}\left( \frac{\kappa \tau _{n}}{2}\right) ^{-2} \\
&\leq &C^{\prime }\frac{n}{\log n}B^{-2rJ_{1n}} \\
&\leq &C^{\prime }\left( \frac{n}{\log n}\right) ^{\frac{1}{2r+1}}.
\end{eqnarray*}%
It follows%
\begin{equation*}
E_{1,1}\leq C_{1,1}\left( \frac{n}{\log n}\right) ^{-\frac{2r}{2r+1}}.
\end{equation*}%
As far as $E_{1,2}$ is concerned, standard calculations using \eqref{Peq} lead to%
\begin{eqnarray*}
E_{1,2} &=&\eta _{n}\sum_{j=0}^{J_{n}}\sum_{q=1}^{Q_{j}}\mathbb{E}\left[
\left\vert \widehat{\beta }_{jq;sK_{n}}-\beta _{jq;sK_{n}}\right\vert ^{2}%
\mathds{1}_{\left\{ \left\vert \widehat{\beta }_{jq;sK_{n}}-\beta
_{jq;sK_{n}}\right\vert \geq \kappa \tau _{n}/2\right\} }\right] \\
&\leq &C\eta _{n}\sum_{j=0}^{J_{n}}\sum_{q=1}^{Q_{j}}\mathbb{E}^{\frac{1}{2}}%
\!\left[ \left\vert \widehat{\beta }_{jq;sK_{n}}-\beta _{jq;sK_{n}}\right\vert
^{4}\right] \! \mathbb{P}^{\frac{1}{2}}\!\left[ \left\vert \widehat{\beta }%
_{jq;sK_{n}}-\beta _{jq;sK_{n}}\right\vert \!\geq \!\frac{\kappa \tau _{n}}{2}%
\right] \\
&\leq &C^{\prime }\sum_{j=0}^{J_{n}}B^{j}n^{-1}n^{-\frac{\delta }{2}}\leq
C^{\prime \prime }B^{J_{n}}n^{-1}n^{-\frac{\delta }{2}}\leq C_{1,2}\eta
_{n}\left( \log n\right) ^{-1}n^{-\frac{\delta }{2}}.
\end{eqnarray*}%
On the other hand, we obtain%
\begin{eqnarray*}
E_{1,3} &=&\eta _{n}\sum_{j=0}^{J_{n}}\sum_{q=1}^{Q_{j}}\left\vert \beta
_{jq;sK_{n}}\right\vert ^{2}\mathbb{E}\left[ \mathds{1}_{\left\{ \left\vert 
\widehat{\beta }_{jq;sK_{n}}-\beta _{jq;sK_{n}}\right\vert \geq \kappa \tau
_{n}/2\right\} }\right] \\
&\leq &C_{1,3}n^{-\delta }\left\Vert F\right\Vert _{L^{2}\left( \mathbb{S}%
^{1}\right) }^{2}.
\end{eqnarray*}%
Finally, we get%
\begin{eqnarray*}
E_{1,4} &\leq &C\eta _{n}\sum_{j=0}^{J_{n}}\sum_{q=1}^{Q_{j}}\left\vert
\beta _{jq;sK_{n}}\right\vert ^{2}\mathds{1}_{\left\{ \left\vert \beta
_{jq;sK_{n}}\right\vert <2\kappa \tau _{n}\right\} } \\
&\leq &C\sum_{j=0}^{J_{1,n}}\sum_{q=1}^{Q_{j}}\left\vert 2\kappa \tau
_{n}\right\vert ^{2}+\eta
_{n}\sum_{j=J_{1,n}}^{J_{n}}\sum_{q=1}^{Q_{j}}\left\vert \beta
_{jq;sK_{n}}\right\vert ^{2} \\
&\leq &C^{\prime }\left( B^{J_{1,n}}\left( \frac{n}{\log n}\right)
^{-1}+\sum_{j=J_{1,n}}^{J_{n}}B^{-2rj}\right),
\end{eqnarray*}%
so that%
\begin{equation*}
E_{1,4}\leq C_{1,4}\left( \frac{n}{\log n}\right) ^{-\frac{2r}{2r+1}},
\end{equation*}%
as claimed.
\end{proof}
The next result was originally presented in \cite{bkmpAoSb} as Lemma 16,
hence the proof is here omitted.
\begin{lemma}
\label{probab}Let $\sigma $ be a finite positive constant such that 
\begin{equation*}
\sigma
\geq \left( \left\Vert F\right\Vert _{L^{\infty }\left( \mathbb{S}%
^{1}\right) }\left\Vert \psi _{jq;s}\right\Vert _{L^{2}\left( \mathbb{S}%
^{1}\right) }^{2}\right) ^{\frac{1}{2}} . 
\end{equation*}
Then, there exists constants $c_{P},c_{E},C>0$ such that, for $B^{j}\leq \left( \frac{n}{\log n}\right) ^{%
\frac{1}{2}}$, the following inequalities hold
\begin{align}
&\mathbb{P}\left[ \left\vert \widehat{\beta }_{jq;sK_{n}}-\beta
_{jq;sK_{n}}\right\vert >x\right] \leq 2\exp \left( \frac{nx^{2}}{2\left(
\sigma ^{2}+c_{P}xB^{\frac{j}{2}}\right) }\right);  \label{Pineq} \\
&\mathbb{E}\left[ \left\vert \widehat{\beta }_{jq;sK_{n}}-\beta
_{jq;sK_{n}}\right\vert ^{2}\right] \leq c_{E}n^{-1};  \label{Eineq}\\
\label{Peq}
&\mathbb{P}\left[ \left\vert \widehat{\beta }_{jq;sK_{n}}-\beta
_{jq;sK_{n}}\right\vert >\frac{\kappa \tau _{n}}{2}\right] \leq Cn^{-\delta }
\end{align}%
where $\delta \geq $ $6\sigma ^{2}$.
\end{lemma}

\noindent \textbf{Acknowledgements} - The author whishes to thank D. Marinucci and I.Z. Pesenson for the useful suggestions and discussions.



\end{document}